\documentclass{lms}
\usepackage{amsmath}
\usepackage{amssymb}
\usepackage{txfonts}
\usepackage{color}
\usepackage{tikz}

\newtheorem{theorem}{Theorem}[section] 
\newtheorem{lemma}[theorem]{Lemma}     
\newtheorem{corollary}[theorem]{Corollary}
\newtheorem{proposition}[theorem]{Proposition}

\newnumbered{assertion}[theorem]{Assertion}    
\newnumbered{conjecture}[theorem]{Conjecture}  
\newnumbered{convention}[theorem]{Convention}
\newnumbered{definition}[theorem]{Definition}
\newnumbered{hypothesis}[theorem]{Hypothesis}
\newnumbered{remark}[theorem]{Remark}
\newnumbered{note}[theorem]{Note}
\newnumbered{observation}[theorem]{Observation}
\newnumbered{problem}[theorem]{Problem}
\newnumbered{question}[theorem]{Question}
\newnumbered{algorithm}[theorem]{Algorithm}
\newnumbered{example}[theorem]{Example}
\newunnumbered{notation}[theorem]{Notation} 

\newcommand{\ignore}[1]{}

\providecommand{\aut}{\mathop{\rm aut}\nolimits}

\providecommand{\cent}{\mathop{\rm cent}\nolimits}

\providecommand{\colim}{\mathop{\rm colim}\nolimits}

\providecommand{\Cent}{\mathop{\rm Cent}\nolimits}

\providecommand{\enid}{\mathop{\rm End}\nolimits}

\providecommand{\GL}{\mathop{\rm GL}\nolimits}

\providecommand{\id}{\mathop{\rm id}\nolimits}

\providecommand{\Ker}{\mathop{\rm ker}\nolimits}

\providecommand{\lub}{\mathop{\rm lub}\nolimits}

\providecommand{\Map}{\mathop{\rm map}\nolimits}

\providecommand{\ord}{\mathop{\rm ord}\nolimits}

\providecommand{\pt}{\mathop{\rm pt}\nolimits}
\providecommand{\pr}{\mathop{\rm pr}\nolimits}
\providecommand{\inclu}{\mathop{\rm inc}\nolimits}
\providecommand{\quot}{\mathop{\rm quot}\nolimits}
\providecommand{\res}{\mathop{\rm res}\nolimits}
\providecommand{\rk}{\mathop{\rm rk}\nolimits}

\providecommand{\saut}{\mathop{\rm saut}\nolimits}

\providecommand{\sign}{\mathop{\rm sign}\nolimits}
\providecommand{\Stab}{\mathop{\rm stab}\nolimits}

\providecommand{\slope}{\mathop{\rm slope}\nolimits}

\providecommand{\SL}{\mathop{\rm SL}\nolimits}
\providecommand{\Sym}{\mathop{\rm sym}\nolimits}
\providecommand{\tors}{\mathop{\rm tors}\nolimits}
\providecommand{\tr}{\mathop{\rm tr}\nolimits}

\providecommand{\vol}{\mathop{\rm vol}\nolimits}
\providecommand{\Verti}{\mathop{\rm vert}\nolimits}

\newcommand{\FJC}{{\bf{FJC}}}
\newcommand{\Fin}{{\mathcal{F}\text{in}}}

\newcommand{\VCyc}{{\mathcal{V}\mathcal{C}\text{yc}}}
\newcommand{\VSol}{{\mathcal{V}\mathcal{S}\text{ol}}}

\newcommand{\IN}{\mathbb{N}}

\newcommand{\IQ}{\mathbb{Q}}
\newcommand{\IR}{\mathbb{R}}

\newcommand{\IZ}{\mathbb{Z}}
\newcommand{\cala}{\mathcal{A}}

\newcommand{\calf}{\mathcal{F}}

\newcommand{\calg}{\mathcal{G}}
\newcommand{\calh}{\mathcal{H}}

\newcommand{\calu}{\mathcal{U}}
\newcommand{\calv}{\mathcal{V}}

\newcommand{\calw}{\mathcal{W}}

\newcommand{\frl}{\mathfrak{L}}

\newcommand{\x}{{\times}}

\newcommand{\e}{{\varepsilon}}
\newcommand{\CAT}{{\operatorname{CAT}}}
\newcommand{\FS}{\mathit{FS}}

\newcommand{\ev}{\mathit{ev}}

\title[The Farrell-Jones conjecture for S-arithmetic groups]
 {The Farrell-Jones conjecture for $S$-arithmetic groups} 

\author{H. R\"uping}

\classno{18F25 (primary), 19A31, 19B28, 19G24(secondary).
}

\extraline{This work was supported by the Leibniz Prize of Prof. Dr. Wolfgang L\"uck granted by the Deutsche Forschungsgemeinschaft.
}

\begin{document}
\maketitle

\begin{abstract}
This paper contains the results of my PhD-thesis. I will show the $K$- and $L$-theoretic Farrell-Jones conjecture (\FJC) for the groups $\GL_n(F(t))$ where $F$ and $\GL_n(\mathbb{Q})$. This especially implies the conjecture for all $S$-arithmetic groups. 
\end{abstract}
\section{Introduction}

The Farrell-Jones conjecture makes predictions about the structure of the algebraic $K$-theory of group rings. There is an an $L$-theoretic version. It implies a lot of well-known conjectures such as the Bass-, Borel- and Novikov-conjecture.
In this paper I will prove the \FJC~for all groups that are linear over $F(t)$ for a finite field $F$ (Theorem~\ref{ICGlnfullyreduced}) or over $\IQ$. This means groups that are subgroups of $\GL_n(F[t][S^{-1}])$ or of $\GL_n(\IQ)$. 

The action of $\GL_n(\IZ)$ on its symmetric space has been used to show the Farrell-Jones in \cite{Bartels-Lueck-Reich-Rueping(2012KandL)}.  I will extend these methods to show the Farrell-Jones conjecture for groups which are linear over $\IQ$  or $F(t)$ for a finite field $F$. This includes in particular all $S$-arithmetic groups.

I will show the strongest version of this conjecture for those groups; the version with coefficients in any additive category with a group action and with finite wreath products. This version has strong inheritance properties, for example any group commensurable to a subgroup of one of the groups mentioned above will satisfy the \FJC. 
This paper is based on my thesis \cite{ruping2013farrell}.

\tableofcontents

\section{Axiomatic setting}\label{sec:axiomaticsetting}

Let me first recall some definitions. Most of them can be found in \cite{Bartels-Lueck(2012CAT(0)flow)}.
Let $X$ be a proper, finite dimensional $\CAT(0)$ space with a proper, isometric group action of a group $G$.
The flow space of $X$  is the set of all maps $\IR\rightarrow X$ whose restriction to some interval $[a,b]$ with $-\infty\le a\le b\le \infty$ is a geodesic and which are locally constant on the complement of this interval 
\[d(f,g):=\int_\IR \frac{d_X(f(t),g(t))}{2e^{|t|}}dt\]
defines a metric on this flow space. $\Phi_t(f):=f(\_+t)$ defines an $\IR$-action (\emph{flow}) on $FS(X)$ that commutes with the induced $G$-action

Let $\calf$ be a family of subgroups of $G$. Let us say that $FS(X)$ \emph{admits long and thin $\calf$-covers}, if there is a $N>0$ such that we can find for every $R>0$ an $G$-invariant open cover $\calu$ and an $\varepsilon>0$ such that for every point $x$ there is a open set $U\in \calu$ containing $B_\varepsilon(\varphi_{[-R,R]}(x))$.

The goal of this chapter is to formulate and proof Proposition~\ref{prop:1:coversAtInfinity}. It will apply to those general linear groups in consideration. Bartels and L\"uck have defined in \cite[Definition~0.4]{Bartels-Lueck(2012CAT(0)flow)} when a group $G$ is transfer reducible over a family of subgroups. Furthermore they showed in Proposition 5.11 how a flow space can be used to show transfer reducibility. Wegner defined a notion of strong transfer reducibility and showed that the same setup also gives strong transfer reducibility in \cite[Definition~3.1 and Theorem~3.4]{wegner2012k}.
The argument uses the following properties of the flow space, which already have been verified for those flow spaces mentioned above:\\
\begin{table}[h]\vspace*{-3ex}
\label{tab:implications}
\begin{tabular}{lcr}
Name & Definition & Verified in \\ \hline
niceness of $X$&\cite[Convention~5.1]{Bartels-Lueck(2012CAT(0)flow)} &\cite[Section~6.2]{Bartels-Lueck(2012CAT(0)flow)}\\\hline
contracting transfers &\cite[Definition~5.9]{Bartels-Lueck(2012CAT(0)flow)} &\cite[Section~6.4]{Bartels-Lueck(2012CAT(0)flow)}\\\hline
long $\calf$-covers at infinity  &\cite[Definition~5.5]{Bartels-Lueck(2012CAT(0)flow)}& \emph{only} periodic part in ~\cite[Theorem~4.2]{Bartels-Lueck(2012CAT(0)flow)}
\\
and periodic flow lines&&
\end{tabular}
\end{table}

One part of the niceness is that there is a bound on the order of finite subgroups of $G$. This is not satisfied in this setting. There is a workaround and the same statements hold even without this assumption see (\cite[Theorem~4.3]{mole2013equivariant}). 

Finally there is also a ``almost'' version of transfer reducibility which inherits to wreath products with finite groups in \cite[Definition~5.3]{Bartels-Lueck-Reich-Rueping(2012KandL)}. 
We have the following implications.
\begin{table}[h]\vspace*{-3ex}
\label{it:transfredimpFJ}
\begin{tabular}{lcr}
 Assumption & conclusion& reference \\ \hline
$G$ (almost) (strongly)  & $G\wr F$ is almost (strongly)  &\cite[Theorem~5.1]{Bartels-Lueck-Reich-Rueping(2012KandL)}\\
transfer reducible over $\calf$&transfer reducible over $\calf^\wr$ &\\\hline
$G$ transfer reducible over $\calf$ & \FJC-K up to dim 1 relative $\calf$ & \cite[Theorem~1.1]{bartels2009borel}\\
&and \FJC-L relative $\calf_2$ hold for $G$&\\\hline
$G$ almost transfer reducible & \FJC-K up to dim 1 relative $\calf'$  & \cite[Proposition~5.4]{Bartels-Lueck-Reich-Rueping(2012KandL)}\\
&and \FJC-L relative $\calf'$ hold for $G$&\\\hline
the strong versions & eliminates the dimension restriction &\cite[Theorem~1.1]{wegner2012k}\\ &&\cite[Proposition~5.4]{Bartels-Lueck-Reich-Rueping(2012KandL)}
\end{tabular}
\end{table}\\
Here $\calf^\wr$ denotes the family of those subgroups which are virtually subgroups of some product $\prod_{f\in F} H_f\subset G^F\subset G\wr F$ with $H_f\in \calf$ and $\calf_2$ denotes all subgroups of $G$ which contain a subgroup from $\calf$ of index at most two.

\begin{definition}[Long $\calf$-covers at infinity and periodic flow lines]
    Let $\FS_{\le \gamma}(X)$ be the subspace of $\FS(X)$ of those generalized geodesics $c$  for which 
there exists for every $\epsilon > 0$ an element $\tau \in (0, \gamma + \epsilon]$ and 
 $g \in G$ such that $g\cdot c = \Phi_{\tau}(c)$ holds. We will say that $\FS$ 
    \emph{admits long $\calf$-covers at infinity and periodic flow lines}
    if the following holds:
    
    There is $N > 0$ such that for every $\gamma > 0$ there is a 
    collection $\calv$  of open $\calf$-subsets of $\FS$ and $\e > 0$  
    satisfying:
    \begin{enumerate}
         \item \label{def:at-infty_plus_periods:G-inv}
             $\calv$ is $G$-invariant: $g \in G$, 
             $V \in \calv \implies gV \in \calv$;
         \item \label{def:at-infty_plus_periods:dim}
             $\dim \calv \leq N$;
         \item \label{def:at-infty_plus_periods:covers}
           there is a compact subset $K \subseteq \FS$ such that
           \begin{enumerate}
             \item $\FS_{\leq \gamma} \cap G \cdot K = \emptyset$;
             \item for $z \in \FS \setminus G \cdot K$ there is 
                $V \in \calv$ such that 
                $B_\e(\Phi_{[-\gamma,\gamma]}(z)) \subset V$.
           \end{enumerate}
    \end{enumerate}
\end{definition}

The ``at infinity'' part is automatically satisfied if the group acts cocompactly. However this is not the case here. 

The following proposition sums up all conditions that are used in \cite{Bartels-Lueck-Reich-Rueping(2012KandL)} to prove that the group action of $\GL_n(\IZ)$ on the space of inner products admits long coverings at infinity and periodic flow lines. 

I will show that the general linear group over $ R[S^{-1}]$ where $R$ is either $\IZ$ or $F[t]$ for a finite field $F$ and $S$ is a finite set of primes in $R$ satisfies these conditions. It might be interesting to find other groups which also satisfy these assumptions.

\begin{proposition}\label{prop:1:coversAtInfinity} Let $G$ be a group, $X$ be a $G$-space, $N$ a natural number and $\calw$  a collection of open subsets of $X$ such that
\begin{enumerate}
\item $X$ is a proper $\CAT(0)$ space, 
\item the covering dimension of $X$ is less or equal to $N$,
\item the group action of $G$ on $X$ is proper and isometric,
\item $G\calw \coloneqq\{gW\mid g\in G, W\in \calw\}=\calw$,
\item the sets $gW$ and $W$ are either disjoint or equal for all $g\in G,W\in \calw$,
\item the dimension of $\calw$ is less or equal to $N$. 
\item the $G$ operation on 
\[X\setminus ( \bigcup \calw^{-\beta} )\coloneqq\{x\in X\mid \nexists W\in \calw:\overline{B}_\beta(x)\subset W\}\]
is cocompact for every $\beta \ge 0$.
\end{enumerate}
Then $FS(X)$ admits long $\mathfrak{F}$-covers at infinity and periodic flow lines for the family $\mathfrak{F}\coloneqq\VCyc \cup \{H\le G \mid \exists \;W\in \calw\;\forall \;h\in H:\; hW=W\}$.
\\
As explained above this means that $G$ is strongly transfer reducible over the family $\mathfrak{F}$ and thus $G$ satisfies the $K$-theoretic \FJC~relative to $\mathfrak{F}$ and the $L$-theoretic \FJC~relative to the family $\mathfrak{F}_2$. Further $G\wr F$ satisfies the $K$- and $L$-theoretic \FJC~relative $\mathfrak{F}^\wr$ for any finite group $F$.
\end{proposition}

\begin{proof}
$FS(X)$ is  a proper metric space by~\cite[Proposition~1.9]{Bartels-Lueck(2012CAT(0)flow)}. Hence it is locally compact. Fix $\gamma \ge 1$. Let $\beta\coloneqq 4+\gamma +1$. Pick a compact subset $L\subset X$ such that $G\cdot L=X\setminus \bigcup \calw^{-\beta}$. 
For this compact subset $L$ we obtain a natural number $M$, a real number $\varepsilon > 0$ and a set $\calu$ of subsets of $\FS(X)$ from \cite[Theorem~4.2]{Bartels-Lueck(2012CAT(0)flow)}. We can assume $\varepsilon \le 1$. 
Let $\calv\coloneqq\ev_0^{-1}(\calw)\coloneqq\{\ev_0^{-1}(W)\mid W\in \calw\}$. We have

\begin{enumerate}
\item $\calv$ is a $G$-set with $gV\cap V\in \{\emptyset,V\}$ for any $g\in G$ and any $V\in \calv$,
\item every element $V\in \calv$ is an open subset of $FS(X)$ since the evaluation map is continuous by \cite[Lemma 1.4]{Bartels-Lueck(2012CAT(0)flow)}),
\item the dimension of $\calv$ is bounded by $N$,
\item the group action on $\ev_0^{-1}(X\setminus \calw^{-R})=FS(X) \setminus \ev_0^{-1}(\calw^{-R})$ is cocompact (as the evaluation map is proper \cite[Lemma 1.10]{Bartels-Lueck(2012CAT(0)flow)}).
\end{enumerate}
Consider the union $\calu \cup \calv$. Each element is an open $\VCyc \cup \{H\le G \mid \exists \;W\in \calw\;\forall \;h\in H:\; hW=W\}$-subset. Define
\[S \coloneqq \{c \in \FS(X) \mid \exists \;Z \in \calu \cup \calv \; \text{with}\;\overline{B}_{\epsilon}\bigl(\Phi_{[-\gamma,\gamma]}(c)\bigr) \subseteq Z\}.\]
This set $S$ contains $\FS(X)_{\le \gamma} \cup |\ev_0^{-1}(\calw^{-(5+\gamma)})|$ by the following
argument.  If $c \in |\ev_0^{-1}(\calw^{-R})|$  we get for any $c'\in \overline{B}_{\epsilon}\bigl(\Phi_{[-\gamma,\gamma]}(c)\bigr)$ by \cite[Lemma~3.4]{Bartels-Lueck-Reich-Rueping(2012KandL)} $d(c'(0),c(0))\le 4+\gamma+\varepsilon \le 5+\gamma $ and hence $c'(0)\in W$. So $c'\in \ev_0^{-1}(W)$. So we verified that $|\calw^{-R}|$ is contained in $S$.
If $c \in\FS(X)_{\le \gamma}$ and $c \notin |\ev_0^{-1}(\calw^{-(5+\gamma)})|$, then $c \in \FS(X)_{\le \gamma}$ and $c(0) \in G\cdot L$ and hence $c \in S$ by Theorem~\cite[Theorem~4.2(v)]{Bartels-Lueck(2012CAT(0)flow)}.

Next we prove that $S$ is open. 
Assume that this is not the case. Then there 
exists $c \in S$ and a sequence $(c_k)_{k \ge 1}$ of elements in $\FS(X) - S$
such that $d_{\FS(X)}\bigl(c,c_k\bigr) <1/k$ holds for $k \ge 1$. Choose $Z
\in \calu \cup \calv$ with
$\overline{B}_{\epsilon}\bigl(\Phi_{[-\gamma,\gamma]}(c)\bigr) \subseteq Z$.
Since $\FS(X)$ is proper as metric space
by~\cite[Proposition~1.9]{Bartels-Lueck(2012CAT(0)flow)} and
$\overline{B}_{\epsilon}\bigl(\Phi_{[-\gamma,\gamma]}(c)\bigr)$ has bounded
diameter, $\overline{B}_{\epsilon}\bigl(\Phi_{[-\gamma,\gamma]}(c)\bigr)$ is
compact.  Hence we can find $\mu > 0$ with 
$B_{\epsilon+  \mu}\bigl(\Phi_{[-\gamma,\gamma]}(c)\bigr) \subseteq Z$.  We conclude
from~\cite[Lemma~2.3]{Bartels-Lueck(2012CAT(0)flow)} for all $s \in
[-\gamma,\gamma]$
\[
d_{\FS(X)}\bigl(\Phi_s(c),\Phi_s(c_k)\bigr) \le e^{s} \cdot
d_{\FS(X)}\bigl(c,c_k\bigr) < e^{\tau} \cdot 1/k.
\]
Hence we get for $k \ge 1$
\[
B_{\epsilon}\bigl(\Phi_{[-\gamma,\gamma]}(c_k)\bigr) \subseteq B_{\epsilon+
  e^{\tau} \cdot 1/k}\bigl(\Phi_{[-\gamma,\gamma]}(c)\bigr).
\]
Since $c_k$ does not belong to $S$, we conclude that
$B_{\epsilon+  e^{\tau} \cdot 1/k}\bigl(\Phi_{[-\gamma,\gamma]}(c)\bigr)$ is not contained in $Z$.
This implies $e^{\tau} \cdot 1/k \ge \mu$ for all $k \ge 1$, a contradiction.
Hence $FS(X)-S$ is a closed $G$-subset of the cocompact set $FS(X)-|\calw^{-R}|$. So it is also cocompact and there is a compact $K\subset FS(X)$ with $G\cdot K=FS(X)-S$.
All in all the $G$-system of open sets $\calu\cup\calv$ of dimension $\le M+N+1$ has the following properties
\begin{enumerate}
\item $FS_{\le \gamma}(X)\cap G\cdot K =FS_{\le \gamma}(X)\cap (FS(X)\setminus S)=\emptyset$ as  $FS_{\le\gamma}(X)\subset S$;
\item for $z\in FS(X)\setminus G\cdot K=S$ there is a  $V\in \calv$ such that $B_\varepsilon(\Phi_{[-\gamma,\gamma]}(z))$.
\end{enumerate}
Hence $FS(X)$ admits long $\mathfrak{F}$-covers at infinity and periodic flow lines. This implies strong transfer reducibility over $\calf$ by \cite[Theorem~4.3]{mole2013equivariant} and thus the mentioned references in~\ref{it:transfredimpFJ} yield the conclusions.
\end{proof}

\section{The canonical filtration}\label{sec:canFilt}

This section shows how the systems of open sets used in Proposition~\ref{prop:1:coversAtInfinity} are constructed. The ideas of this section can all be found in \cite{Grayson(1984)}. Let $V$ be a free $\IZ$-module and $s$ an inner product on $\IR\otimes_\IZ V$. The size of submodules can be measured in two different ways -- by its rank and its volume. The desired open sets in the space of homothety classes of inner products are constructed by comparing these two quantities. This section is formulated in a very general way, since the same constructions also apply for the rings $\IZ[S^{-1}],F[t][S^{-1}]$.

An \emph{order-theoretic lattice} $\frl$ is a poset such that any finite subset has a least upper bound and a greatest lower bound. For any two elements $W,W'\in \frl$ let $W+W'$ denote their least upper bound and let $W\cap W'$ denote their greatest lower bound. Let $0$ denote the minimal element. It is the least upper bound of the empty set. Let $1$ denote the maximal element which is the greatest lower bound of the empty set.
\ignore{An order-theoretic lattice $\frl$ can be viewed as a category whose objects are the elements of $\frl$. The morphism set from $W$ to $W'$ consists of an unique element if $W\le W'$ and is empty otherwise. Least upper bounds and greatest lower bounds are product and coproducts.}
\begin{convention}\label{conv:latticeForFiltr}
Let $\frl$ be an order-theoretic lattice.
Suppose furthermore there are functions $\rk: \frl\rightarrow \IN$ and $\log\vol:\frl\rightarrow \IR$ such that
\begin{enumerate}
\item \label{conv:latticeForFiltr_StrMon} $\rk$ is strictly monotone. This means that for all $W,W'\in \frl:$
\[W< W'\Rightarrow \rk(W)<\rk(W').\]
\item \label{conv:latticeForFiltr_rkAdd} $\rk$ is additive. This means that for all $W,W'\in \frl$:
\[\rk(W\cap W')+\rk(W+W')=\rk(W)+\rk(W').\]
\item \label{conv:latticeForFiltr_logovolsubAdd} $\log\vol(-):\frl\rightarrow \IR$ is subadditive. This means that for all $W,W'\in \frl:$
\[\;\log\vol(W\cap W')+\log\vol(W+W')\le \log\vol(W)+\log\vol(W').\]
\item \label{conv:latticeForFiltr_FinShort} For each $C\in \IR$ there are only finitely many $L\in \frl$ with $\log\vol(W)\le C$.
\item \label{conv:latticeForFiltr_normalization}$ \rk(0)=0,\log\vol(0)=0$. 
\end{enumerate}
\end{convention}

\begin{remark}\label{rem:lattices}\begin{enumerate}
\item The strict monotonicity holds for the lattice of direct summands of $\IZ^n$ whereas it fails for the lattice of all submodules of $\IZ^n$.
\item \label{rem:lattices:zwei}In the lattice of direct summands of $\IZ^n$ the least upper bound $V+W$ is \emph{not} the sum of the modules but the direct summand spanned by the sum of the modules, i.e. the preimage of the torsion group of $\IZ^n/(V+W)$.
\item It follows that $0$ and $1$ are the only elements of rank zero resp. $\rk(1)$.
\item Later the volume will also depend on the choice of an inner product. Thus we will view $\log\vol$ as a real valued function on the space of inner products. 
\end{enumerate}
\end{remark}

\begin{definition} We can plot every element $W\in \frl$ on the $(x,y)$-plane with $x$-coordinate equal to its rank and $y$-coordinate equal to $\log\vol(W)$. For any fixed rank  between zero and $\rk(1)$ there is a lowest point among all points with that rank. 

We can omit those elements which lie above or on a line connecting two other points of this set and call the remaining points the \emph{canonical path}. 
\end{definition}

Of course, it might happen that there are several elements from $\frl$ with the same rank and volume. We will see that this will not be the case for the points in the canonical path. 

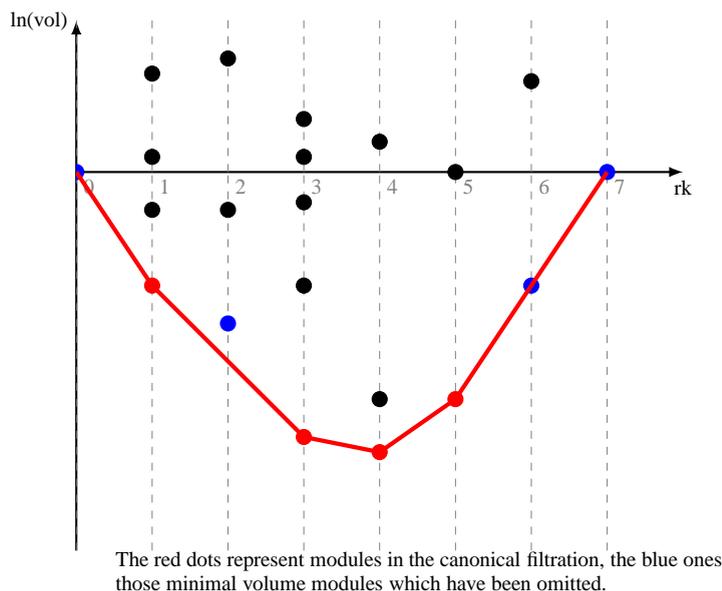
\begin{figure}[ht]
  \centering
  \begin{tikzpicture}
    \coordinate (Origin)   at (0,0);
    \coordinate (XAxisMin) at (0,0);
    \coordinate (XAxisMax) at (8,0);
    \coordinate (YAxisMin) at (0,-5);
    \coordinate (YAxisMax) at (0,2);
    \draw [thick, black,-latex] (XAxisMin) -- (XAxisMax)node[below]{rk};
    \draw [thick, black,-latex] (YAxisMin) -- (YAxisMax)node[left]{ln(vol)};

    \clip (0,-5) rectangle (10cm,2cm); 
   \foreach \x in {0,...,7}{
	   \draw[style=help lines,dashed](\x,-5) -- +(0,5)node [below right] {$\x$};
	   \draw[style=help lines,dashed](\x,0) -- +(0,5);
   }
   \node[draw,circle,inner sep=2pt,fill,blue] at (0,0) {};
   \node[draw,circle,inner sep=2pt,fill,red] at (1,-1.5) {};
   \node[draw,circle,inner sep=2pt,fill,blue] at (2,-2) {};
   \node[draw,circle,inner sep=2pt,fill,red] at (3,-3.5) {};
   \node[draw,circle,inner sep=2pt,fill,red] at (4,-3.7) {};
   \node[draw,circle,inner sep=2pt,fill,red] at (5,-3) {};
   \node[draw,circle,inner sep=2pt,fill,blue] at (6,-1.5) {};
   \node[draw,circle,inner sep=2pt,fill,blue] at (7,0) {};   
   \draw [ultra thick,red] (0,0)--(1,-1.5)--(3,-3.5)--(4,-3.7)--(5,-3)--(7,0);
   \node[draw,circle,inner sep=2pt,fill] at (1,-0.5) {};
   \node[draw,circle,inner sep=2pt,fill] at (1,0.2) {};
   \node[draw,circle,inner sep=2pt,fill] at (1,1.3) {};
   \node[draw,circle,inner sep=2pt,fill] at (2,-0.5) {};
   \node[draw,circle,inner sep=2pt,fill] at (2,1.5) {};
   \node[draw,circle,inner sep=2pt,fill] at (3,-1.5) {};
   \node[draw,circle,inner sep=2pt,fill] at (3,-0.4) {};
   \node[draw,circle,inner sep=2pt,fill] at (3,0.2) {};
   \node[draw,circle,inner sep=2pt,fill] at (3,0.7) {};
   \node[draw,circle,inner sep=2pt,fill] at (4,-3) {};
   \node[draw,circle,inner sep=2pt,fill] at (4,0.4) {};
   \node[draw,circle,inner sep=2pt,fill] at (5,0) {};
   \node[draw,circle,inner sep=2pt,fill] at (6,1.2) {};
  \end{tikzpicture}
     \vspace*{8pt}
   \begin{minipage}{8cm}
   The red dots represent modules in the canonical filtration, the blue ones those minimal volume modules which have been omitted.
   \end{minipage}

  \caption{The canonical plot.}
  \label{figure:canPlot}
\end{figure}

\begin{definition}\label{def:cW}
We can define for $W\in \frl\setminus \{0,1\}$ a number 
\[c_W\coloneqq \inf_{\genfrac{(}{)}{0pt}{}{W_0\lneq W}{W\lneq W_2}} \slope(W_2,W)-\slope(W,W_0),\]
where $\slope(W,W')$ is defined as $\frac{\log\vol(W)-\log\vol(W')}{\rk(W)-\rk(W')}$.
\end{definition}

Note that the denominators of $\slope(W,W')$ are nonzero for $W\lneq W'$ by the strict monotonicity of the rank.
If $W$ represents a vertex in the canonical path we get $c_W>0$. Otherwise $W$ would lie above the edge from $W_0$ to $W_2$. The following lemma leads to the converse.

\begin{lemma}\label{lem:cWneg} Given two incomparable elements $V,W\in \mathfrak{L}$. Then $c_W\le 0$ or $c_V\le 0$.
\end{lemma}
\begin{proof*}
Incomparable means that $V\nleq W$ and $W\nleq V$. Especially $V\cap W$ is smaller than $V$, so it can't be $W$ since $W$ is not smaller than $V$. So we have by the same argument
\[V\cap W\lneq V \lneq\lub(V,W) \mbox{ and }V\cap W\lneq W \lneq\lub(V,W).\]
Let us assume that $c_W>0$. So we have to show that $c_V\le 0$. We have 
\multbox\begin{eqnarray*}
0&<& c_W\\
&=&\inf_{\genfrac{}{}{0pt}{}{W_0\lneq W}{W\lneq W_2}} \slope(W_2,W)-\slope(W,W_0)\\
&\le & \slope(\lub(V,W),W)-\slope(W,V\cap W)\\
&=&\frac{\log\vol(\lub(V,W))-\log\vol(W)}{\rk(\lub(V,W))-\rk(W)}-\frac{\log\vol(W)-\log\vol(V\cap W)}{\rk(W)-\rk(V\cap W)}\\
&\stackrel{\mbox{\ref{conv:latticeForFiltr}\ref{conv:latticeForFiltr_rkAdd},\ref{conv:latticeForFiltr_logovolsubAdd} }}\le &\frac{\log\vol(V)-\log\vol(V\cap W)}{\rk(V)-\rk(V\cap W)}-\frac{\log\vol(\lub(V,W))-\log\vol(V)}{\rk(\lub(V,W))-\rk(V)}\qquad \\
&=&-(\slope(\lub(V,W),V)-\slope(V,V\cap W))\\
&\le &-\inf_{\genfrac{}{}{0pt}{}{V_0\lneq V}{V\lneq V_2}} \slope(V_2,V)-\slope(V,V_0)\\
&=&-c_V.\end{eqnarray*}\emultbox
\end{proof*}

\begin{figure}[ht]
  \centering
  \begin{tikzpicture}
    \coordinate (Origin)   at (0,0);
    \coordinate (XAxisMin) at (-1,0);
    \coordinate (XAxisMax) at (5,0);
    \coordinate (YAxisMin) at (0,0);
    \coordinate (YAxisMax) at (0,2);

    \clip (-1,-0.3) rectangle (7cm,2cm); 
   \foreach \x in {0,...,5}{
	   \draw[style=help lines,dashed](\x,-5) -- +(0,10);
   }
   \node[draw,circle,inner sep=1pt,fill] at (2,0.2){};
   \node[draw,circle,inner sep=1pt,fill] at (0,0.7){};
   \node[draw,circle,inner sep=1pt,fill] at (3,1.5){};
   \node[draw,circle,inner sep=1pt,fill] at (5,0.5){};
   \draw [thin, black] (2,0.2) -- (2,0.2)node[below]{$W$};
   \draw [thin, black] (0,0.7) -- (0,0.7)node[above]{$V\cap W$};
   \draw [thin, black] (3,1.5) -- (3,1.5)node[above]{$V$};
   \draw [thin, black] (5,0.5) -- (5,0.5)node[right]{$\lub(V,W)$};
   \draw [thin, black](0,0.7) --(2,0.2);
   \draw [thin, black](0,0.7) --(3,1.5);
   \draw [thin, black](5,0.5) --(3,1.5);
   \draw [thin, black](5,0.5) --(2,0.2);
   \draw [thin, black,dotted](5,0.9) --(2,0.2);
   \draw [thin, black,dotted](5,0.9) --(3,1.5);

  \end{tikzpicture}
       \vspace*{8pt}
   \begin{minipage}{8cm}
Subadditivity gives an upper bound for the logarithmic volume of $\lub(V,W)$ indicated by the dotted line. It completes a the parallelogram so subadditivity can also be called the ``parallelogram rule''.
   \end{minipage}
  \label{figure:filtrLemProof}
\end{figure}

\begin{corollary}\label{cor:cWAndCanFiltr} We have that
\begin{enumerate}
\item every vertex in the canonical path is represented by a unique element $V\in \frl$ and
\item those elements form a chain.
\item Furthermore an element $V\in \frl\setminus \{0,1\}$ represents a vertex in the canonical path if and only if $c_V>0$.
\end{enumerate}
\end{corollary}
\begin{proof*}
\begin{enumerate}
\item By definition any element $V\in \frl$ has $c_V>0$. The slope of the outgoing line must
be larger than the slope of the incoming line. 
Given two elements $V,V'\in \frl \setminus\{0,1\}$ that represent the 
same vertex of the canonical path. They cannot be incomparable by the last lemma. So either $V\le V'$ or $V'\le V$.
But they have the same rank. So the strict monotonicity of the rank gives $V=V'$.
\item Let $V_0,\ldots,V_m$ be the list of elements ordered by rank that represent the vertices of the canonical path whose rank is at least one and  at most $\rk1-1$. By the last item the ranks of those elements are all distinct. 
So $\rk(V_i)<\rk(V_j)$ if $i<j$. As in the last item we know that either $V_i\le V_j$ or $V_j\le V_i$.
The monotonicity of the rank gives that $V_i\le V_j$ for $i<j$. So
\[0 \lneq V_0 \lneq \ldots \lneq V_m \lneq 1\]
is a chain.
\item Given $V\in \frl$ with $c_V>0$. Assume $V$ does not represent a vertex in the canonical filtration.
So $V$ lies above a line segment of the canonical path. Say that segment starts at $W$ and ends at $W'$.
By the last lemma we again know that $V$ and $W$ (resp. $V$ and $W'$) are not incomparable. As in the last item we get
$W\le V\le W'$. Because $V$ lies above the edge from $W$ to $W'$ we have
\multbox\begin{eqnarray*}
0&<& \slope(W',V)-\slope(V,W)\\
&\le& \inf_{\genfrac{(}{)}{0pt}{}{W_0\lneq W}{W\lneq W_2}} \slope(W_2,W)-\slope(W,W_0)=c_V.\end{eqnarray*}\emultbox
\end{enumerate}
\end{proof*}

\begin{definition} The chain of elements $0=V_0\le V_1\le \ldots V_m =1$ that represent the vertices in the canonical path is called the \emph{canonical filtration} of $(\frl, \rk, \log\vol)$.
\end{definition}

\section{Volume: The integral case}\label{sec:volint}
\begin{definition}\label{def:nf:volume}
Given an inner product on $\IR^n$ and a submodule $M\subset \IZ^n$. Then we can define its volume as
\[\vol_M(s):=\det((s(m_i,m_j))_{1\le i,j\le \rk(M)})^\frac{1}{2}.\]
\end{definition}
This definition (together with the usual rank) satisfies all properties needed in Convention~\ref{conv:latticeForFiltr}.

\begin{proposition}\label{prop:nflatticeForFiltrSatisfied} Let $V$ be a finitely generated free $\IZ$-module and $s$ an inner product on $\IR\otimes_\IZ V$. Consider the lattice of direct summands of $V$. Let $\lub(W,W')$ denote the least upper bound of two elements $W,W'\in \frl$.
The logarithmic volume function $W\mapsto \ln\vol_W(B)$ and the rank $W\mapsto \rk_{\IZ}(W)$ have the following properties. 
\begin{enumerate}
\item  $\rk$ is strictly monotone, i.e. $\rk(W)<\rk(W')$ for all $W,W'\in \frl$ with $W<W'$.
\item $\rk$ is additive, i.e. $\rk(W\cap W')+\rk(\lub(W,W'))=\rk(W)+\rk(W')$ for all $W,W'\in \frl$.
\item  The function $\ln\vol(-):\frl\rightarrow \IR$ is subadditive. This means
\[\ln\vol_{W\cap W'}(s)+\ln\vol_{\lub(W,W')}(s)\le \ln\vol_W(s)+\ln\vol_{W'}(s)\mbox{ for all }W,W'\in \frl.\]
\item  For each $C\in \IR$ there are only finitely many $L\in \frl$ with $\ln\vol_W(s)\le C.$
\item $ \rk(0)=0,\ln\vol_0(s)=0.$ 
\end{enumerate}
\end{proposition}
\begin{proof}
The elementary proof can be found in \cite[Proposition~4.6]{ruping2013farrell}. The only tricky part is the subadditivity for which there is also a proof in \cite[Theorem~1.12]{Grayson(1984)} with the minor difference that the symbol $W+W'$ denotes there the sum of submodules and not the direct summand spanned by that sum. This is a little bit stronger since the volume of the sum is larger.
\end{proof}

\section{Volume: The function field case}\label{sec:FFvol}

Let $F$ be a finite field and consider the ring $Z\coloneqq F[t]$ and its quotient field $Q$. Let us examine the valuation 
\[\nu: Q\rightarrow\IZ\cup \{\infty\},\qquad \frac{p}{q}\mapsto \deg(q)-\deg(p).\]
We use the convention that the degree of the zero polynomial is $-\infty$.  Its valuation ring is 
\[R \coloneqq \{\frac{p}{q}\in Q\mid \deg(p)\le \deg(q)\}=\{x\in Q\mid\nu(x)\ge 0\}.\]
The following definition is the analogue of a ``lattice'' from \cite[section~1]{Grayson(1984)} in the integral case; but I would like to avoid this term because it will also appear with different meanings.

\begin{definition}\label{def:ffVolSpace} A $Z$-volume space $(V,S)$ is a finitely generated free $Z$-module $V$ with the choice of an $R$-lattice $S$ in $Q\otimes_Z V$. This means that $S$ is a finitely generated $R$-submodule with $\rk_Z(V)=\rk_R(S)$. 
\end{definition}
It is not hard to see that such an $S$ is torsionfree and hence isomorphic to $R^{\rk(V)}$ by the structure theorem.

\begin{definition}\label{def:ffquotVolSpace} We say that $(W,S')$ is a \emph{sub-volume space} of $(V,S)$ (written $(W,S')\subset (V,S)$) if $W\subset V$ is a $Z$-submodule and $S'=S\cap Q\otimes_Z W\eqqcolon \res_W(S)$ for the inclusion $i:W\hookrightarrow V$.

If $V/W$ is projective the \emph{quotient volume space} of $(W,S')\subset (V,S)$ is defined as \[(V/W,S/(S\cap (Q\otimes_Z W))).\] Let us denote it by $(V,S)/(W,S')$. Let $\quot_W(S)\coloneqq S/(S\cap (Q\otimes_Z W))$.
\end{definition} 

\begin{remark}\label{rem:ffEasyvolProps} We have the following easy properties: \begin{enumerate}
\item \label{rem:ffEasyvolPropsEins} The $R$-lattice occuring in the definition of sub-volume space can be omitted. More precisely, any submodule $W$ of $V$ can be turned into a sub-volume space with the choice $\res_W(S)$. 
\item \label{rem:ffEasyvolPropsDrei} The $R$-module $S\cap (Q\otimes_Z W)$ is a direct summand in $S$. 
\item Consequently there is an $R$-basis $b_1,\ldots,b_n$ of $S\subset Q\otimes_Z V$ with \[\langle b_1,\ldots,b_{\rk(W)}\rangle_Q=Q\otimes_Z W.\]
\item \label{rem:ffEasyvolPropsZwei} The quotient volume space is a volume space.
\item{}[Analog of {\cite[Lemma~1.1]{Grayson(1984)}} ]\label{lem:ffelemPropVol} Let $L=(V,S)$ be a volume space.
For $V_1\subset V_2 \subset V$ we have $\res_{V_1}\circ \res_{V_2}(S) =\res_{V_1}(S)$.
\end{enumerate}
\end{remark}

\begin{remark} We have  to be a bit careful here when we speak about subquotients. It makes a difference whether you first pass to quotients and then to a subobject or the other way round. For example if $V$ is the free $F[t]$-module on generators $e_1,e_2$ and let $V_1$ be the submodule spanned by $1e_1+t^ne_2$ and let $V_2$ be the submodule spanned by $e_2$. 
Let $S$ be the $R$-module spanned by $1\otimes e_1,1\otimes e_2$. 
In this example we have 
\[\quot_{V_1\cap V_2}(\res_{V_1}(S))\neq \res_{V_1/(V_1\cap V_2)}\quot_{V_2}(S).\]
\end{remark}
The situation is better if we assume that $V_2$ is a submodule of $V_1$:
\begin{lemma}\label{lem:ffsubquotients}
Let $L=(V,S)$ be a volume space and let $V_2, V_1 \subset V$ be submodules such that $V/V_2$ is projective. We have:
\begin{enumerate}
\item \label{lem:ffsubquotients:eins}\[\quot_{V_1\cap V_2}(\res_{V_1}(S)) \subset \res_{V_1+V_2/V_2}(\quot_{V_2}(S)).\]
\item \label{lem:ffsubquotients:zwei}If $V_2\subset V_1$, then both sides are equal:
\[\quot_{V_2}(\res_{V_1}(S)) = \res_{V_1/V_2}(\quot_{V_2}(S)).\]
\item If additionally $V/V_1$ is also projective, then it is the quotient of $V/V_2$ by $V_1/V_2$, i.e.
\[\quot_{V_1/V_2}(\quot_{V_2}(S))=\quot_{V_1}(S).\]
\end{enumerate}
\end{lemma}
\begin{proof*}
Clearly $V_1/(V_1+V_2)$ is a submodule of $V/V_2$. So let us now compare the lattices. Let 
\[\pr:Q\otimes_Z V\rightarrow Q\otimes_Z V/(S\cap Q\otimes_Z V_2)\] denote the projection. 
\begin{multline*}
\quot_{V_2\cap V_1}(\res_{V_1}(S))
\coloneqq(S\cap Q\otimes_Z V_1)/(S\cap Q\otimes_Z (V_1\cap V_2))
= \pr(S\cap Q\otimes_Z V_1)\\
\subseteq  \pr(S)\cap \pr(Q\otimes_Z V_1)
=S/(S\cap Q\otimes_Z V_2)\cap (Q\otimes_Z (V_1+V_2)/V_2)
\eqqcolon \res_{(V_1+V_2)/V_2}(\quot_{V_2}(S)).\end{multline*}
This proves the first claim. Note that $\pr(A\cap B)=\pr(A)\cap \pr(B)$ holds if $A$ and $B$ are $\pr$-saturated, i.e. $\pr(\pr^{-1}(A))=A$. In the case where $A$ is a submodule this means that $S\cap Q\otimes_Z V_2\subset A$. This is where the condition $V_2\subset V_1$ enters. This proves the second claim. The third claim follows from
\multbox\begin{eqnarray*}
\quot_{V_1/V_2}(\quot_{V_2}(S))
&=&(S/(S\cap Q\otimes_Z V_2))/((S\cap Q\otimes_Z V_1)/(S\cap Q\otimes_Z V_2))\\
&=& S/(S\cap Q\otimes_Z V_1)= \quot_{V_1}(S).\end{eqnarray*}\emultbox
\end{proof*}

\begin{remark} Let $(V,S)$ be a volume space. Then $(\Lambda^m V,\Lambda^m S)$ is a volume space. 
\end{remark}
If $(V,S)$ is a volume space we get two ways of choosing a basis for $\langle V\rangle_Q=\langle S\rangle_R$, first by picking a $Z$-Basis of $V$ and second by picking a $R$-basis of $S$. The difference is interesting:

\begin{definition}[logarithmic volume]\label{def:ffvol} Let $(V,S)$ be a volume space. Pick an $Z$-basis $v_1,\ldots,v_n$ for $V$ and an $R$-basis $b_1,\ldots,b_n$ of $S$. The $Q$-vector space $Q\otimes_Z \Lambda^n V\cong \Lambda^n (Q\otimes_Z V)$ is one dimensional. Consider the element $q\in Q$ with $v_1\wedge\ldots\wedge v_n =q(b_1\wedge \ldots \wedge b_n)$ (It exists since $b_1\wedge\ldots\wedge b_n\neq 0$). Define 
\[\log\vol_V(S)\coloneqq-\nu(q).\]
\end{definition}
Clearly the volume is independent of the involved choices. Choosing different bases will change $q$ by a multiplication with an element in $Z^*=F^*$ resp. $R^*=\{q\in Q\mid\nu(q)=0\}$. This change does not affect the valuation.

\ignore{\begin{remark}\label{lem:ffequivarianceOfVol} We clearly have the following equivariance property. For any $\varphi \in \aut_Z(V)$ we have 
\[\log\vol_{\varphi(W)}(\varphi(S))=\log\vol_{W}(S).\]
\end{remark}}

Let us now find a way to compute the volume of a sub-volume space $W\subset (V,S)$ without constructing a basis for $S\cap Q\otimes_Z W$.

\begin{proposition}[Formula for the logarithmic volume]\label{prop:ffVolFormula} Let $(V,S)$ be an $Z$-volume space and let $(b_1,\ldots,b_n)$ be an $R$-basis of $S$.
Let $W\subset Q\otimes V$ be a finitely generated $Z$-submodule. Choose a $Z$-basis $w_1,\ldots,w_m$ of $W$. The set 
\[\{b_{i_1}\wedge\ldots\wedge b_{i_m}\mid 1\le i_1<\ldots <i_m\le n \}\]
is a $Q$-basis for $\Lambda^m Q\otimes_Z V$. Write $w_1\wedge \ldots\wedge w_m$ as a linear combination of this basis
\[w_1\wedge \ldots\wedge w_m= \sum_{1\le i_1<\ldots <i_m\le n} \lambda_{i_1,\ldots,i_m}\cdot b_{i_1}\wedge\ldots\wedge b_{i_m}.\]
Then the logarithmic volume of $W$ with respect to $S$ is
\[\log\vol_W(\res_W(S))= \sup_{1\le i_1<\ldots <i_m\le n}-\nu(\lambda_{i_1,\ldots ,i_m}).\]
\end{proposition}
\begin{proof}
First we show that the right hand side does not depend on the choice of bases. This is clear for the basis $w_1,\ldots,w_m$. If we would choose another basis $b'_*$ the new coefficients $\lambda'_*$ would be a linear combination of the old ones whose factors are products of the entries of the base change matrix. Those entries lie in $R$ and thus have valuation $\ge 0$. 

Because $F[t]$ is a principal ideal domain $W$ is again free. Furthermore $R$ is  also a principal ideal domain; so the same holds for the $R$-module $S$. So the minimal valuation of those coefficients cannot decrease. Since we can swap the roles of $b$ and $b'$ it also cannot increase and so it has to be the same. If we extend bases of $W$ (resp. $S/S\cap(Q\otimes_Z W)$) to $V$ (resp. $S$) we can achieve that only of the coefficients $\lambda_*$ is nonzero and we end up with the definition of $\log\vol_W(S\cap Q\otimes_ZW)$.
\end{proof} 

\begin{remark} We will sometimes use the abbreviation $\log\vol_W(S)$ for $\log\vol_W(\res_W(S))$. 
\end{remark}

\begin{remark}\label{rem:ffvolumeOfZero} The logarithmic volume of the zero volume space is defined as zero.
\end{remark}

\begin{remark}\label{rem:ffieldFormulaVolume} There is an explicit formula for the logarithmic volume. Let $S$, $W$, $(b_i)_i$, $(w_i)_i$, $m$, $n$ be as above and let $w_i=\sum_{j=1}^n \lambda_{i,j} b_j$ with $\lambda_{i,j}\in Q$. 
Inserting this in the definition of the logarithmic volume yields
\[\log\vol_W(\res_W(S))= \sup_{1\le i_1<\ldots <i_m\le n} -\nu(\sum_{\sigma \in \Sigma_n} \sign(\sigma)\lambda_{\sigma(1),{i_1}}\cdot \ldots\cdot \lambda_{\sigma(m),i_m})\]
This means the following. We consider the non-square matrix $(\lambda_{i,j})_{i,j}$. Consider all $m\times m$ minors, i.e. square matrices obtained from this matrix by deleting rows/columns. The logarithmic volume of $\langle w_1,\ldots,w_m\rangle$ is the negative of the minimum of the valuation of their determinants.
\end{remark}

\begin{lemma}\label{lem:ffSubmoduleOfFiniteIndex}Let $(V,S)$ be a volume space and let $W'\subset W\subset V$ be a chain of $Z$-modules of the same rank $m$. Let $A$ denote
a matrix that represents the inclusion after choice of bases for $W$ and $W'$. Then 
\[\log\vol_{W'}(\res_{W'}(S))= \log\vol_W(\res_W(S))+(-\nu(\det(A))).\]
Furthermore $-\nu(\det(A))=\dim_F(W/W')$.
\end{lemma}
\begin{proof} Choose bases $w_1,\ldots,w_m$ of $W$ and $w'_1,\ldots,w'_m$ of $W'$ (invariant factor theorem).  We obtain by definition of the determinant:
\[ w'_1\wedge \ldots \wedge w'_m=\det(i)(w_1\wedge\ldots \wedge w_m)\]
and hence $\log\vol_{W'}(S)=-\nu(\det(i))+\log\vol_W(S)$.
The equality $-\nu(\det(A))=\dim_F(W/W')$ follows directly from the invariant factor theorem.
\end{proof}

\begin{lemma}[Volume of a quotient]\label{lem:ffVolQuot} Let $(V,S)$ be a volume space and let $(W,S\cap (Q\otimes_Z W))$ be a sub-volume space such that $V/W$ is projective. Then
\[\log\vol_V(S)=\log\vol_W(\res_W(W))+\log\vol_{V/W}(\quot_W(S)).\]
 \end{lemma}
\begin{proof}
Again we extend bases of $W$ (resp. $S\cap (Q\otimes_Z W)$) to the whole of $V$ (resp. $S$). Then the matrix in Remark~\ref{rem:ffieldFormulaVolume} has block form and hence we have to compute the valuation of its determinant. The upper left contributes $\log\vol_W(\res_W(W))$ and the lower right block contributes $\log\vol_{V/W}(\quot_W(S))$.
\end{proof}

\begin{lemma}[Parallelogram constraint/subadditivity]\label{prop:ffSubaddidivity} Let $(V,S)$ be a volume space and let $W_1,W_2$ be finitely generated $Z$-submodules of $V$. Then
\[\log\vol_{W_1\cap W_2}(S)+\log\vol_{W_1+W_2}(S) \le \log\vol_{W_1}(S)+\log\vol_{W_2}(S).\]
\end{lemma}
\begin{proof} Using Lemma~\ref{lem:ffSubmoduleOfFiniteIndex} we can first replace $W_1,W_2$  by the direct summands generated by them. 
Then $W_1\cap W_2$ is also a direct summand. After passing to quotients by $W_1\cap W_2$ we can further assume that $W_1\cap W_2=0$. Now we have to be a bit more careful with the notation:
\begin{eqnarray*}
&&\log\vol_{W_1}(\res_{W_1}(S))+\log\vol_{W_2}(\quot_{W_1\cap W_2}\res_{W_2}(S))\\
&\ge &\log\vol_{W_1}(\res_{W_1}(S))+\log\vol_{(W_1+W_2)/W_1}(\res_{(W_1+W_2)/W_1}\quot_{W_1}(S))\\
&=&\log\vol_{W_1}(\res_{W_1}(S))+\log\vol_{(W_1+W_2)/W_1}(\quot_{W_1}\res_{W_1+W_2}(S))\\
&=&\log\vol_{W_1+W_2}(S)\\
\end{eqnarray*}
The first steps uses Lemma~\ref{lem:ffsubquotients}\ref{lem:ffsubquotients:eins} and the observation that if $S\subset S'$, then $\log\vol_?(S)\le \log\vol_?(S')$. The second step uses ~\ref{lem:ffsubquotients}~\ref{lem:ffsubquotients:zwei} to swap back. The third step is Lemma~\ref{lem:ffVolQuot}.
\end{proof}

\begin{lemma}\label{lem:ffShortFin} Given a volume space $(V,S)$ and a real number $C$. Then there are only finitely many elements $v\in V\setminus \{0\}$ with $\log\vol_{\langle v\rangle_Z}(S)\le C$.
\end{lemma}
\begin{proof}
To simplify notation let $V=Z^n$. Let $b_1,\ldots,b_n$ be a $R$-basis of $S\subset Q^n$. By Remark~\ref{rem:ffieldFormulaVolume} set of such vectors whose logarithmic volume is less than some number $D$ is given by 
\[X=Z^n\cap \{\sum \lambda_ib_i|\lambda_i\in R,-\nu(\lambda_i)<D \}.\]
Let $D':=\min_{i,j=1}^n, -\nu(b_{i,j})$ where $b_{i,j}$ denotes the $j$-th entry of $b_i\in Q^n$. Since $\nu$ is a valuation we get $-\nu(x)<D+D'$. But this means that
$X$ is a subset of $\{p\in F[t] |\deg(p)<D+D'\}^n$ and hence finite.
\end{proof}

\begin{corollary}\label{cor:fffinShort} Given a volume space $(V,S)$ and a real number $C$. Then there are only finitely many submodules $W\subset V$ with $\log\vol_W(\res_W(S))\le C$.
\end{corollary}
\begin{proof}
This follows directly from the previous result and the following claim:

For every $m\le n\coloneqq \rk(V)$ and every $v\in\Lambda^m V$ there are only finitely many submodules $W\subset V$ such that $v=w_1\wedge \ldots \wedge w_m$ for a $Z$-basis $w_1,\ldots,w_m$ of $W$.
Consider $W':=\Ker(\wedge v:V\rightarrow \Lambda^{m+1}V)$. Every such $W$ is a submodule of $W'$ and they all have the same index.  But there are only finitely many submodules of $W'\cong F[t]^m$ of a given index.
\end{proof}

\begin{proposition}[diagonal bases] Let $(V,S)$ be a volume space. Then there is an $R$-basis $b_1,\ldots,b_n$ of $S$ and an $Z$-basis $w_1,\ldots,w_n$ of $V$ such that
$w_i=t^{r_i}b_i$ for some $r_i\in \IZ$ with $r_1 \le r_2\le \ldots \le r_n$.
\end{proposition}
\begin{proof}
The proof is done by induction on $\rk(V)$ and there is nothing to show in the case of $\rk(V)=0$.

Let $v\in V$ be a shortest nontrivial vector. Hence $\langle v\rangle$ is a direct summand of $V$. Let $b_1$ be a basis vector of the $R$-module $(Q\otimes_Z \langle v\rangle_Z)\cap S$. Hence $w_1$ is of the form $\lambda b_1$ for some $0\neq  \lambda \in Q$. 
Without loss of generality we can assume that $\lambda$ is of the form $t^{r_1}$  --- otherwise replace $b_1$ by $\lambda\cdot t^{\nu(\lambda)}b_1$. We get the following two split exact sequences:
\[0\rightarrow \langle v\rangle_Z \rightarrow V\rightarrow V/\langle v\rangle_Z\rightarrow  0,\]
\[0\rightarrow S\cap (Q\otimes_Z\langle v\rangle_Z) \rightarrow S\rightarrow  S/S\cap (Q\otimes_Z\langle v\rangle_Z) \rightarrow  0.\]
By induction we already get such bases for the quotient volume space. Let $b_2,\ldots,b_n$ be preimages of the basis of $S/(Q\otimes_Z \langle v\rangle)\cap S$ and let $w_2,\ldots,w_n\in W$ be preimages of the basis of $W/\langle v\rangle_Z$ under the projection map. We get the following linear combinations
\[w_1=t^{r_1}b_1,\qquad w_i = s_i b_1 + t^{r_i}b_i.\]
for some $s_i \in Q$. Let us consider a fixed $i\in \{2,\ldots,n\}$.
Now we can write $s_i$ as a mixed fraction $s_i = \sum_{j=r_1}^m  a_j t^j + t^{r_1-1}\cdot r$ with $r\in R$. If we now improve the choice of the preimages by replacing $w_i$ by $w_i-\sum_{j=r_1}^m a_j t^{j-r_1}w_1$ we may assume $s_i = t^{r_1-1}\cdot r$. Now we have to use the fact that $v$ was chosen to be a shortest vector and the formula~\ref{prop:ffVolFormula} for the computation of the volume to get
\[r_1 = \log\vol_{\langle w_1\rangle_Z}(S)\le \log\vol_{\langle w_i\rangle_Z}(S)= \max(-\nu(s_i),r_i).\] 
We have already achieved that $-\nu(s_i) = (r_1-1)-\nu(r)\le (r_1-1)$. Hence $\frac{s_i}{t^{r_i}}\in R$ and $r_1\le r_i$. If we finally replace $b_i$ by $b_i+\frac{s_i}{t^{r_i}}b_1$, we can  assume that $s_i=0$. Hence we have found a basis of the desired form.
\end{proof}

Let us understand the cosets of the $\aut_Z(V)$-action on the set of all $R$-lattices in $Q\otimes_Z V$. 

\begin{proposition}\label{prop:ffieldIsomandCanFilt} The numbers $r_1,\ldots,r_n$ from the last proposition uniquely determine the $\aut_Z(V)$-orbit. Furthermore the canonical filtration of $(V,S)$ consists exactly of the modules $\{\langle \{w_i\mid r_i\le C\} \rangle_Z \mid C\in \IZ\}$. The integral volume of such a module $\langle \{w_i\mid r_i\le C\} \rangle_Z$ is just $\sum_{i\in \{j\mid r_j\le C\}}r_i$. 

Furthermore $c_{\langle w_1,\ldots,w_m\rangle}(S) = r_{m+1}-r_m$.

Unlike in the integral case there is in every dimension a module on the canonical path; this can be seen as an implication of the ultrametric inequality.
\end{proposition}
\begin{proof}
An element $f\in \aut_Z(V)$ maps such bases again to such bases. So we only have to show that the numbers $r_1,\ldots,r_n$ do not depend on the choices. The idea is to express $r_m$ intrinsically as the difference of minimal logarithmic volume of a rank $m+1$ and a rank $m$ direct summand.
A easy computation shows that $\sum_{i=0}^m r_i=\log\vol_{\langle w_1,\ldots,w_m}(S)$ and that every other rank $m$ submodule cannot have smaller logarithmic volume.
Let us have a look at the canonical path.
The slope does not decrease and it increases at rank $m$ if and only if $r_{m+1}>r_m$. Thus the modules $\{\langle \{w_i\mid r_i\le C\} \rangle_Z \mid C\in \IZ\}$ are really the canonical filtration of $(V,S)$. 
The value of $c_{\langle w_1,\ldots w_m\rangle}$ is $r_{m+1}-r_m$ because $r_{m+1}$ is the slope of the canonical path from $m$ to $m+1$ and $r_m$ is the slope from $m-1$ to $m$. 
\end{proof}

\begin{figure}[ht]
  \centering
  \begin{tikzpicture}
    \coordinate (Origin)   at (0,0);
    \coordinate (XAxisMin) at (0,0);
    \coordinate (XAxisMax) at (8,0);
    \coordinate (YAxisMin) at (0,-3);
    \coordinate (YAxisMax) at (0,2);
    \draw [thick, black,-latex] (XAxisMin) -- (XAxisMax)node[below]{rk};
    \draw [thick, black,-latex] (YAxisMin) -- (YAxisMax)node[left]{ln(vol)};

    \clip (0,-3) rectangle (10cm,2cm); 
   \foreach \x in {0,...,7}{
	   \draw[style=help lines,dashed](\x,-5) -- +(0,5)node [below right] {$\x$};
	   \draw[style=help lines,dashed](\x,0) -- +(0,5);
   }
   \node[draw,circle,inner sep=2pt,fill,blue] at (0,0) {};
   \node[draw,circle,inner sep=2pt,fill,blue] at (1,-1) {};
   \node[draw,circle,inner sep=2pt,fill,red] at (2,-2) {};
   \node[draw,circle,inner sep=2pt,fill,red] at (3,-2.5) {};
   \node[draw,circle,inner sep=2pt,fill,blue] at (4,-2) {};
   \node[draw,circle,inner sep=2pt,fill,blue] at (5,-1.5) {};
   \node[draw,circle,inner sep=2pt,fill,red] at (6,-1) {};
   \node[draw,circle,inner sep=2pt,fill,blue] at (7,0) {};   
   \draw [ultra thick,red] (0,0)--(2,-2)--(3,-2.5)--(4,-2)--(5,-1.5)--(6,-1)--(7,0);
  \end{tikzpicture}
  \vspace*{8pt}
   \begin{minipage}{8cm}
	 Sketch of the canonical plot of a volume space with $r_*=(-2,-2,-1,1,1,1,2)$. The slope of the canonical path from $i-1$ to $i$ is $r_i$.
   \end{minipage}
  \label{figure:canPlotAndTheNumbersR}
\end{figure}

\begin{lemma}[Some trivia] Let $(V,S), (V,S')$ be two volume spaces. 
\begin{itemize}
\item If $S\subset S'$, then $\log\vol_W(\res_W{S})\le \log\vol_W(\res_W(S')).$
\item \label{cor:FFieldvolumeOfNeighbors} Let $W\subset V$ be any submodule. Let us assume $S\subset S'\subset tS$ ($t$ is the variable in $F[t]=Z$). Then
\[\log\vol_W(\res_W(S)) \le \log\vol_W(\res_W(S')) \le \rk_Z(W)+\log\vol_W(\res_W(S)).\]
\end{itemize}
\end{lemma}

So we have shown that the logarithmic volume function satisfies all conditions from Convention~\ref{conv:latticeForFiltr}:

\begin{proposition} \label{prop:fflatticeForFiltr} Let $(V,S)$ be a volume space. Consider the lattice $\frl$ of direct summands of $V$. 
The logarithmic volume function $W\mapsto \log\vol_W(\res_W(S))$ and the rank $W\mapsto \rk_{F[t]}(W)$ have the following properties. 
\begin{enumerate}
\item \label{prop:fflatticeForFiltr_StrMon} $\rk$ is strictly monotone, i.e.  $\rk(W)<\rk(W')$ for all $W,W'\in \frl$ with $W<W'$.
\item \label{prop:fflatticeForFiltr_rkAdd} $\rk$ is additive, i.e. $\rk(W\cap W')+\rk(\lub(W,W'))=\rk(W)+\rk(W')$ for all $W,W'\in \frl$.
\item \label{prop:fflatticeForFiltr_logovolsubAdd} The function $\log\vol(-):\frl\rightarrow \IR$ is subadditive. This means that for all $W,W'\in \frl$
\begin{multline*} \log\vol_{W\cap W'}(\res_{W\cap W'}(S))+\log\vol_{\lub(W,W')}(\res_{\lub(W,W')}(S))\\ \le \log\vol_{W}(\res_W(S))+\log\vol_{W'}(\res_{W'}(S)).\end{multline*}
\item \label{prop:fflatticeForFiltr_FinShort} For each $C\in \IR$ there are only finitely many $L\in \frl$ with \\
$\log\vol_W(\res_W(S))\le C$.
\item \label{prop:fflatticeForFiltr_normalization}$ \rk(0)=0,\log\vol(0)=0$. 
\end{enumerate}
\end{proposition}
\begin{proof}
\begin{enumerate}
\item This is clear, note that strictness holds since we only consider direct summands.
\item This is almost the classical additivity of the rank applied to the short exact sequence \[0\rightarrow W\cap W'\rightarrow W\oplus W'\rightarrow W+W'\rightarrow 0,\]
except that we have to replace $W+W'$ by $\lub(W,W')$. Passing to a finite index submodule does not change the rank. 
\item Proposition~\ref{prop:ffSubaddidivity} shows a stronger statement (with $\lub(W,W')$ replaced by $W+W'$). It is stronger by Lemma~\ref{lem:ffSubmoduleOfFiniteIndex}.
\item This has been done in Corollary~\ref{cor:fffinShort}.
\item This is clear from the definitions (and Remark~\ref{rem:ffvolumeOfZero}). 
\end{enumerate}\end{proof}
\begin{remark}\label{rem:ff:homothetyInvOfcW} It follows directly from Definition~\ref{def:ffvol} that
\[\log\vol_W(qS)=\rk(W)\cdot \nu(q)+\log\vol_W(S)\]
for $q\in Q$.
The function $c_W$ from Definition~\ref{def:cW} is defined as an infimum over functions of the form
\[S\mapsto \frac{\log\vol_{W_2}(S)-\log\vol_W(S)}{\rk(W_2)-\rk(W)}-\frac{\log\vol_{W}(S)-\log\vol_{W_0}(S)}{\rk(W)-\rk(W_0)}.\]
Using the upper formula we see that replacing $S$ by $qS$ does not affect $c_W$. Hence $c_W(S)=c_W(qS)$.
\end{remark}

\section{Volume: The localized case}\label{sec:volloc}
Now let start to study groups of the form $\GL_n(\IZ[\frac{1}{2}])$. Again we want to assign to an inner product its volume in a $\GL_n(\IZ[\frac{1}{2}])$ invariant way. But the volume of the parallelepiped spanned by a $\IZ[\frac{1}{2}]$-basis of $\IZ[\frac{1}{2}]^n$ depends on the choice of this basis. The solution is to add additional structure that tells us which bases are allowed. The set of possible choices for this additional information also carries a $\GL_n(\IZ[\frac{1}{2}])$-action, so it makes sense to pick the volume in a $\GL_n(\IZ[\frac{1}{2}])$-invariant way.

\begin{convention}Let
\begin{itemize}
\item $Z$ denote either the integers $\IZ$ or the polynomial ring $F[t]$ for a finite field $F$,
\item an element $z\in Z$ be called \emph{normalized} if it is positive in the case of $Z=\IZ$ resp. if its leading coefficient is one in the case of $F[t]$.
\item $\mathfrak{P}$ denote the set of all normalized primes in $Z$,
\item $T\subset  \mathfrak{P}$ denote a finite subset,
\item $Q$ denote the quotient field of $Z$,
\item $Z[T^{-1}]$ be the ring $\{\frac{a}{b}\in Q\mid \nu_p(\frac{a}{b})\ge 0\mbox{ for all }p\in \mathfrak{P}\setminus T\}$,
\item $Z_T$ be the ring $Z[(\mathfrak{P}\setminus T)^{-1}]$,
\item $n$ be a fixed integer,
\item $(z,T)$ denote the product of all normalized prime factors of $z\in Z$ that lie in $T$,
\item $\ord(m)$ denote a generator of the ideal  $Ker(Z\rightarrow M \quad r\mapsto rm)$ for an element $m$ of a $Z$-module.
\end{itemize}
\end{convention}  

\begin{remark} \label{rem:propZSinv}
\begin{enumerate}
\item Every nonzero ring element $z\in Z$ is associated to a unique normalized element. 
\item The rings $Z,Z_T,Z[T^{-1}]$ are all Euclidean rings and hence principal ideal domains. For $\IZ$ a degree function is given by the absolute value and for $F[t]$ it is given by the degree of a polynomial. A degree function on $Z[S^{-1}]$ is for example given by 
\[\frac{a}{b}\mapsto \deg((a,\mathfrak{P}\setminus S)) - \deg((b,\mathfrak{P}\setminus S))\]
where $\deg$ denotes a degree function on $Z$.
\end{enumerate}
\end{remark}

\begin{definition}\label{def:IntegralStructure} A \emph{integral structure} with respect to $T$ on a finitely generated free $Z[T^{-1}]$-module $V$ of rank $n$ is a finitely generated $Z_T$-submodule $B$ of $Q\otimes_{Z[T^{-1}]} V$ of rank $n$.
\end{definition}

\begin{remark}\label{rem:GLnQactsOnIntegralStructures} Let $V\coloneqq Z[T^{-1}]^n$. We have the following trivia:
\begin{enumerate}
\item The underlying $Z_T$-module of an integral structure is always free of rank $n$ by the structure theorem for finitely generated modules over a PID.
\item $\aut_Q(Q \otimes V)\cong \GL_n(Q)$ acts transitively on the set of all integral structures on $V$; for any two integral structures $B,B'$ we can pick $Z_T$-bases and a matrix $A\in \GL_n(Q)$ that maps one basis to the other. The stabilizer of the standard integral structure $Z[(\mathfrak{P}\setminus T)^{-1}]^n\subset Q^n$ is $\GL_n(Z[(\mathfrak{P}\setminus T)^{-1}])$. Hence every other stabilizer is conjugate to $\GL_n(Z[(\mathfrak{P}\setminus T)^{-1}])$ in $\GL_n(Q)$.
\item \label{rem:QnmodBisTorsion} For any integral structure $B$ we get that $Q^n/B\cong Q^n/Z[(\mathfrak{P}\setminus T)^{-1}]^n$ is $T$-torsion.
\end{enumerate}
\end{remark}

The poset of direct summands of $Z^n$ is isomorphic to the poset of subvector spaces of $Q^n$; the isomorphisms are given by $\_\cap Z^n$ and $\langle \_\rangle_Q$. A more fancy version of this is the following proposition:

\begin{proposition}\label{prop:isomPoset} Let $B$ be an integral structure on $V$. Then $V\cap B$ is a finitely generated, free $Z$-module and $\_\cap B$ defines a (rank-preserving) isomorphism from the poset of direct summands of $V$ to the poset of direct summands of $V\cap B$.
Thus it also preserves greatest lower bounds and least upper bounds and hence is an isomorphism of lattices.
\end{proposition}
\begin{proof}
Finite generation of $V\cap B$ follows from the following elementary criterion: A $Z[S^{-1}]$-submodule $M$ of $Q^n$ is finitely generated, iff it is contained in a submodule of the form $\frac{1}{z}Z[S^{-1}]^n$ for some $z\in Z$. 

 By the criterion above we find some $z\in Z$ such that $(d,\mathfrak{P}\setminus T)$ divides $z$ for all denominators of entries of elements of $V\subset Q^n$ and a $z'$ such that $(d,T)$ divides $z'$ for all denominators of entries of elements of $B\subset Q^n$. Thus all $V\cap B$ is a submodule of $\frac{1}{zz'}Z^n$ and hence finitely generated.
 
It is totally elementary to verify that $\_ \cap B$ and $\langle\_\rangle_{Z[T^{-1}]}$ are inverse to each other. The rank of a direct summand is an intrinsic property of the poset of direct summands -- it is the length of the longest ascending chain ending with this direct summand. Thus it is preserved under a poset isomorphism.
\end{proof}

\begin{lemma} Given two disjoint sets of primes. An integral structure on $Z[(T_1\cup T_2)^{-1}]^n$ is uniquely determined by an integral structure on $Z[T_1^{-1}]^n$ and an integral structure on $Z[T_2^{-1}]^n$.
\end{lemma}
\begin{proof}
The proof works exactly the same way: From an integral structure $B\subset Q^n$ with respect to $T_1\cup T_2$ we can obtain $(\langle B\rangle_{Z_{T_1}},\langle B\rangle_{Z_{T_2}})$. Conversely, given two such integral structures we can take their intersection. 
An elementary computation shows that both compositions are the identity.
\end{proof}
Equivalence classes of integral structures relative to a single prime form the vertices of an affine building. With the last lemma we can identify integral structures (up to rescaling) with the vertices of a product of buildings.

\subsection{The localized case}

\begin{convention}Let 
\begin{itemize}
\item $n\in \IN$ be a fixed non-negative integer.
\item $Z$ be either $\IZ$ (integral case) or $F[t]$ (function field case) for a finite field $F$,
\item $V$ be a finitely generated free $Z[T^{-1}]$-module of rank $n$,
\item $\tilde{X}(V)$ denote the set of all inner products on $\IR\otimes_\IZ V$ in the integral case or the set of all $\{\frac{a}{b}\in Q\mid\deg(b)\ge\deg(a)\}$-lattices in $Q \otimes_Z V$ for a finitely generated, free $Z[T^{-1}]$-module in the function field case.
\item $\frl$ denote the order-theoretic lattice of direct summands of the $Z[T^{-1}]$-module $V$.
\end{itemize}

\end{convention}

Now we are ready to define the volume function.
\begin{definition}\label{def:locVol} Let $\tilde{Y}_T(V)$ denote the set of all integral structures on $V$ relative to $T$.
 Define the logarithmic volume function of $V$ as
\[\log\vol: \frl \times \tilde{X}(V)\times \tilde{Y}_T(V) \rightarrow \IR; \qquad (W,s,B)\mapsto \log\vol_{W\cap B}(s)\eqqcolon \log\vol_W(B,s).\]
\end{definition}

\begin{remark}\label{rem:loc:equivarianceOfVol} We have the following trivia:\begin{enumerate}
\item\label{rem:loc:equivarianceOfVol:eins} An element in $\tilde{Y}_T(Z[T^{-1}]^n)$ is just a choice of an equivalence class of a system of $n$ linear independent vectors in $Q^n$, where two such systems are equivalent if and only if their $Z[(\mathfrak{P}\setminus T)^{-1}]$-span agrees. Hence 
\[\tilde{Y}_T(Z[T^{-1}]^n)\cong \GL_n(Q)/\GL_n(Z[(\mathfrak{P}\setminus T)^{-1}])\]
as left-$\GL_n(Q)$-sets.
\item\label{rem:loc:equivarianceOfVol:zwei} Note that $W\cap B$ is just a finitely generated free $Z$-module. In the integral case $s$ is an inner product on $\IR\otimes_\IZ V$ and hence it can be restricted to 
$W\cap B\subset V\subset \IR\otimes_\IZ V$.

In the function field case $s$ is a lattice in $Q \otimes_Z V$. The inclusion $V\cap B\rightarrow V$  induces an isomorphism $Q\otimes_Z (V\cap B)\rightarrow Q\otimes_Z V = Q\otimes_{Z[T^{-1}]}V$ since $Q\otimes_Z\_$ is exact. So $s$ can also be considered as a lattice in $V\cap B$. Hence $(V\cap B,s)$ is a volume space. So the definition of volume (Definition~\ref{def:ffvol}) for the function field case can be used here.
\item\label{rem:loc:equivarianceOfVol:drei} For any $\varphi \in \aut_{Z[S^{-1}]}(V)$, any submodule $W\subset V$ and any integral structure $B$ we get:
\[\log\vol_{\varphi(W)}((\varphi\cdot s),\varphi\cdot B)=\log\vol_{\varphi(W\cap B)}(\varphi\cdot s)\\
=\log\vol_{W}(s,B).\]
\end{enumerate}
\end{remark}

\subsection{Properties of the volume function for $Z[T^{-1}]$}

Fix an integral structure $B$ on $V$ relative to a set of primes $T$ and an element $s\in \tilde{X}(V)$. We want to show that the function
\[\log\vol_?(s,B):\frl \rightarrow \IR\]
satisfies all conditions from Convention~\ref{conv:latticeForFiltr} so that we can consider the canonical filtration.  

\begin{proposition}\label{prop:loc:convtrue}
Let $\frl$ denote the order-theoretic lattice of direct summands of $V$ and for $W\in \frl$ let $\rk(W)$ denote the $Z[T^{-1}]$-rank of $W$. Let $\log\vol_W(s,B)$ denote the logarithmic volume as above.
We have:
\begin{enumerate}
\item $\rk$ is strictly monotone, i.e. $\rk(W)<\rk(W')$ for all $W,W'\in \frl$ with $W\subsetneq W'$.
\item $\rk$ is additive, i.e. $\rk(W\cap W')+\rk(\lub(W,W'))=\rk(W)+\rk(W')$ for all $W,W'\in \frl$.
\item The function $\log\vol_{-}(s,B):\frl\rightarrow \IR$ is subadditive. This means that for all $W,W'\in \frl$
\[\log\vol_{W\cap W'}(s,B)+\log\vol_{\lub(W,W')}(s,B)\le \log\vol_W(s,B)+\log\vol_{W'}(s,B).\]
\item For each $C\in \IR$ there are only finitely many $L\in \frl$ with $\log\vol_W(s,B)\le C$.
\item $rk(0)=0,\log\vol(0)=0$. 
\end{enumerate}
\end{proposition}
\begin{proof}
\begin{enumerate}
\item This follows from the structure theorem of finitely generated modules over a principal ideal domain.
\item  Since $Q$ is a flat $Z$-module, we can apply $Q\otimes_z\_$ to the following short exact sequences 
\[0\rightarrow W+W'\rightarrow \lub(W,W')\rightarrow \lub(W,W')/(W+W')\rightarrow 0,\]
\[0\rightarrow W\cap W'\rightarrow W\oplus W'\rightarrow W+W'\rightarrow 0.\]
Recall that the rank is defined as $\dim_Q(Q\otimes_Z\_)$. Since $\lub(W,W')/(W+W')$ is torsion (Remark~\ref{rem:lattices}\ref{rem:lattices:zwei}) and $\dim_Q$ is additive we get the result.
\item Using Definition~\ref{def:locVol} of the volume function, we really have to show that:
\[\log\vol_{W_1\cap B}(s)+\log\vol_{W_2\cap B}(s)\ge \log\vol_{W_1\cap W_2\cap B}(s)+\log\vol_{\lub(W_1,W_2)\cap B}(s).\]
This equation just involves the definition of the volume of a $Z$-module. By Proposition~\ref{prop:isomPoset} we get $\lub(W_1,W_2)\cap B=\lub(W_1\cap B,W_2\cap B)$. 
We have already shown
\[\log\vol_{W_1\cap B}(s)+\log\vol_{W_2\cap B}(s)\ge \log\vol_{W_1\cap W_2\cap B}(s)+\log\vol_{\lub(W_1\cap B,W_2\cap B)}(s)\]
in Proposition~\ref{prop:nflatticeForFiltrSatisfied} for the integral case and in Proposition~\ref{prop:fflatticeForFiltr}\ref{prop:fflatticeForFiltr_logovolsubAdd} for the function field case.
\item The definition of the volume function (Definition~\ref{def:locVol}) says that it is just the composition of the old volume function for $Z$ and this isomorphism of lattices. Hence the statement follows directly from the statement for $Z$ (see Proposition~\ref{prop:nflatticeForFiltrSatisfied} for the integral case and Corollary~\ref{lem:ffShortFin} for the function field case) and Proposition~\ref{prop:isomPoset}.
\item The zero module is the minimal element in the lattice and its rank is zero and its logarithmic volume is defined to be zero.
\end{enumerate}
\end{proof}

\begin{remark}\label{rem:loccWIdent} So we can use section~\ref{sec:canFilt} to get for each $W\in \frl$ a number $c_W(s,B)$. 
We have $c_W(s,B)=c_{W\cap B}(S)$. This follows easily from the definitions of both sides and Proposition~\ref{prop:isomPoset}.
\ignore{\begin{eqnarray*}
&&c_W(s,B)\\
&\coloneqq &\inf_{\genfrac{(}{)}{0pt}{}{W_0\subsetneq W}{W\subsetneq W_2}} \frac{\log\vol_{W_2}(s,B)-\log\vol_W(s,B)}{\rk(W_2)-\rk(W)}-\frac{\log\vol_W(s,B)-\log\vol_{W_0}(s,B)}{\rk(W)-\rk(W_0)}\\
&\coloneqq &\inf_{\genfrac{(}{)}{0pt}{}{W_0\subsetneq W}{W\subsetneq W_2}} \frac{\log\vol_{W_2}(s,B)-\log\vol_W(s,B)}{\rk(W_2)-\rk(W)}-\frac{\log\vol_W(s,B)-\log\vol_{W_0}(s,B)}{\rk(W)-\rk(W_0)}\\
&\coloneqq &\inf_{\genfrac{(}{)}{0pt}{}{W_0\subsetneq W}{W\subsetneq W_2}} \frac{\log\vol_{W_2\cap B}(s)-\log\vol_{W\cap B}(s)}{\rk(W_2\cap B)-\rk(W\cap B)}-\frac{\log\vol_{W\cap B}(s)-\log\vol_{W_0\cap B}(s)}{\rk(W\cap B)-\rk(W_0\cap B)}\\
&\coloneqq &\inf_{\genfrac{(}{)}{0pt}{}{W_0\subsetneq W}{W\subsetneq W_2}} \frac{\log\vol_{W_2\cap B}(s)-\log\vol_{W\cap B}(s)}{\rk(W_2)-\rk(W)}-\frac{\log\vol_{W\cap B}(s)-\log\vol_{W_0\cap B}(s)}{\rk(W)-\rk(W_0)}\\
&\eqqcolon & c_{W\cap B}(s).\end{eqnarray*}
This used that the map $-\cap B$ from the lattice of direct summands of the $Z[S^{-1}]$-module $V$ to the lattice of direct summands of $V\cap B$ is a rank-preserving isomorphism by Proposition~\ref{prop:isomPosets}.}
\end{remark}

Furthermore we have the following properties:

\begin{lemma} In the number field case  ($Z=\IZ$) we have
\begin{enumerate}
\item $\vol_W(\lambda s,B)=\lambda^{\rk W} \vol_W(s,B)$ for $\lambda \in \IR,\lambda>0$,
\item $\vol_W(s,p B)=p^{\rk W}\vol_W(s,B)$ for any $p\in T$.
\end{enumerate}
In the function field case  ($Z=F[t]$) we have
\begin{enumerate}
\item for $\lambda \in Z[T^{-1}]\setminus \{0\}$ that
\[\log\vol_W(\lambda S,B) = -\rk(W)\cdot \nu(\lambda) + \log\vol_W(S,B),\]
\item $\log\vol_W(S,pB)=-\rk(W)\nu(p)+\log\vol(S,B)$ for any $p\in T$.
\end{enumerate}
\end{lemma}
\begin{proof*}
We get in the number field case:
\begin{enumerate}
\item \begin{centering}$\vol_W(\lambda s,B)\coloneqq \vol_{W\cap B}(\lambda s)=\lambda^{\rk W} \vol_{W\cap B}(s) \eqqcolon  \lambda^{\rk W} \vol_{W}(s,B).$\end{centering}
The equality in the middle follows directly from the definition of the volume (see Definition~\ref{def:nf:volume}).
\item As $W$ is a $\IZ[T^{-1}]$ module we get $pW=W$ and hence 
$W\cap pB=pW\cap pB = p(W\cap B)$ and consequently
\begin{multline*}
\vol_W(s,pB)=\vol_{p(W\cap B)}(s)=[W\cap B:p(W\cap B)]\vol_{W\cap B}(s)\\
=p^{\rk(W\cap B)}\vol_{W\cap B}(s)=p^{\rk W}\vol_W(s,B).\end{multline*}
\end{enumerate}
Let us now consider the function field case:
\begin{enumerate}
\item We can use the same chain of equalities as in the number field case
\begin{multline*}
\log\vol_W(\lambda S,B)\coloneqq  \log\vol_{W\cap B}(\lambda S)=-\rk(W)\cdot \nu(\lambda) + \log\vol_{W\cap B}(S)\\ 
\eqqcolon  -\rk(W)\cdot \nu(\lambda) + \log\vol_W(S,B)\end{multline*}
and the middle equality is given by Lemma~\ref{lem:ffSubmoduleOfFiniteIndex}.
\item As $W$ is a $Z[T^{-1}]$ module we get $tW=W$ and hence 
$W\cap pB=pW\cap pB = p(W\cap B)$ and consequently
\multbox\begin{multline*}
 \log\vol_W(S,pB)=\log\vol_{p(W\cap B)}(S)\stackrel{\ref{lem:ffSubmoduleOfFiniteIndex}}{=}\dim_F((W\cap B)/p(W\cap B))+\log\vol_{W\cap B}(S)\\
=\rk(W)\deg(p)+\log\vol_{W\cap B}(S)=-\rk(W)\nu(p)+\log\vol_{W\cap B}(S).\end{multline*}\emultbox
\end{enumerate}
\end{proof*}

\begin{corollary}\label{cor:locVolOfAdjacentVert} Given two integral structures $B,B'$ such that $zB \subset B'\subset B$ for some $z\in Z$. Since $B$ is a $Z[\mathfrak{P}\setminus T]$-module we get $pB=B$ for any $p\in \mathfrak{P}\setminus T$. Thus we can leave out all prime factors of $z$ from $\mathfrak{P}\setminus T$. So let us assume that no element of $\mathfrak{P}\setminus T$ divides $z$. We have
\begin{itemize}
\item in the number field case
\[\rk(W)\cdot \ln(z)+\vol_W(s,B)=\ln\vol_W(s,zB)\ge\ln\vol_W(s,B')\ge \ln\vol_W(s,B),\]
\item in the function field case 
\[-\rk(W)\cdot \nu(z) + \log\vol_W(s,B)=\log\vol_W(s,zB)\ge\log\vol_W(s,B')\ge \log\vol_W(s,B).\]
\end{itemize}
\end{corollary}

\begin{corollary}[Scaling invariance of $c_W$]\label{cor:loccWscalingInv} We get in the number field case:
\begin{enumerate}
 \item\label{cor:loccWscalingInv:eins} $c_W(\lambda s,B)=c_W(s,B)$ for any $\lambda \in \IR,\lambda >0$
\item\label{cor:loccWscalingInv:zwei} $c_W(s, pB)=c_W(s,B)$ for any $p\in T$.
\end{enumerate}
and in the function field case
\begin{enumerate}
\item\label{cor:loccWscalingInv:drei} $c_W(\lambda s,B) =c_W(s,B)$ for any $\lambda \in F[t]\setminus \{0\}$
\item\label{cor:loccWscalingInv:vier} $c_W(\lambda s,pB) =c_W(s,B)$ for any $p \in T$
\end{enumerate}
\end{corollary}
\begin{proof} This follows immediately from the definition of $c_W$  (Definition~\ref{def:cW}) and the previous lemma. 
\end{proof}

\begin{definition} Let $X(\IR^n)$ denote the quotient of $\tilde{X}(V)$ under the group action
\[(\IR^+,*)\times \tilde{X}(\IR^n)\rightarrow \tilde{X}(\IR^n)\qquad (\lambda, s)\mapsto \lambda s.\]
Let $T$ be a set of primes. Let $Y_T(n)$ denote the quotient of $\tilde{Y}_T(n)$ under the group action
 of the group of units in $Z[T^{-1}]^*=\cent(\aut_{Z[S^{-1}]}(V))$. By centrality we still have a $\aut_{Z[S^{-1}]}(V)$-action on $Y_T(V)$.
\end{definition}

\begin{remark}
The scaling invariance from Corollary~\ref{cor:loccWscalingInv} shows that the function $c_W$ descends to a function
\[c_W:X( \IR^n)\times Y_T(n)\rightarrow \IR.\]
\end{remark}
The following lemma will be needed to study the action of $\GL_n(Q)$ on a specific $\CAT(0)$-space.

\begin{lemma} \label{lem:matrixFactorizations}
Let $T$ be a set of primes.
\begin{enumerate}

\item\label{lem:matrixFactorizations:eins} Every matrix $A \in \GL_n(Q)$ can be written as a product of a matrix in $\GL_n(Z[T^{-1}])$ and a matrix in $\GL_n(Z[(\mathfrak{P}\setminus T)^{-1}])$. 
\item\label{lem:matrixFactorizations:zwei} Every matrix $A \in \SL_n(Q)$ can be written as a product of a matrix in $\SL_n(Z[T^{-1}])$ and a matrix in $\SL_n(Z[(\mathfrak{P}\setminus T)^{-1}])$. 
\item\label{lem:matrixFactorizations:drei} Furthermore if a subgroup $G$ is conjugate to $\SL_n(Z[(\mathfrak{P}\setminus T)^{-1}])$ in $\GL_n(Q)$ we can also decompose any matrix $A\in \SL_n(\IQ)$ as a product of a matrix in $\SL_n(Z[T^{-1}])$ and a matrix in $G$.
\end{enumerate}
\end{lemma}
\begin{proof}
\begin{enumerate}
\item This is obvious for diagonal matrices. If $A$ is not diagonal let $m$ be the least common multiple of the denominators of all entries of $M$. By the invariant factor theorem applied to the matrix $mA\in M_n(Z)$ we can find integral matrices $B,C,D\in M_n(Z)$ such that $B,C$ are invertible matrices of determinant one and $D$ is a diagonal matrix and $mA=BDC$. Hence $A=B\cdot (\frac{1}{m}D)\cdot C$. Then we apply this lemma to the diagonal matrix $\frac{1}{m}D$ to obtain the result.
\item The product of the two determinants of the two matrices obtained like in the last item is one. One of them lies in $Z[(\mathfrak{P}\setminus T)^{-1}]$ and the other one lies in $Z[T^{-1}]$. Hence they both have to be one.
\item Assume $G=B\cdot \SL_n(Z[(\mathfrak{P}\setminus T)^{-1}])\cdot B^{-1}$. We can first decompose $B=B'B''$ like in the first item. Especially we get $B\cdot \SL_n(\IZ[(\mathfrak{P}\setminus T)^{-1}])\cdot B^{-1}=B'\cdot \SL_n(\IZ[(\mathfrak{P}\setminus T)^{-1}])\cdot B'^{-1}$. Hence without loss of generality we may assume $B\in \GL_n(\IZ[T^{-1}])$. We decompose $B^{-1}AB$ as in the second item and conjugate each factor with $B$. This gives the desired decomposition.
\end{enumerate}
\end{proof}

\begin{proposition}\label{prop:locSLcofinite} For any finite set of primes $S$ the group action of $\saut_{Z[S^{-1}]}(V)$ on $Y_S(V)$ is cofinite.
\end{proposition}
\begin{proof*}Since $\GL_n(Q)\cong \aut_Q(Q\otimes_{Z[S^{-1}]}V)$ acts transitively on $\tilde{Y}_S(V)$ with stabilizer $\GL_n(Z_S)$, we get $\tilde{Y}_S(V)\cong \GL_n(Q)/\GL_n(Z_S)$. By Lemma~\ref{lem:matrixFactorizations}\ref{lem:matrixFactorizations:eins} we also have a transitive group action of $\GL_n(Z[S^{-1}])$. The isomorphism
\begin{multline*}\GL_n(Z[S^{-1}])/SL_n(Z[S^{-1}])\cdot \{\lambda I_n| \lambda \in Z[S^{-1}]\}\stackrel{\det}{\rightarrow} Z[S^{-1}]^*/\{\lambda^m|\lambda\in Z[S^{-1}]^*\}\\{\rightarrow} Z^*\times (\IZ/m)^{|T|}.\end{multline*}
Thus $SL_n(Z[S^{-1}])\cdot \{\lambda I_n| \lambda \in Z[S^{-1}]\}$ has finite index and hence there are only finitely many elements in  
\[SL_n(Z[S^{-1}])\cdot \{\lambda I_n| \lambda \in Z[S^{-1}]\}\backslash \GL_n(Q)/\GL_n(Z_S)\]
By centrality this is the same as 
\[\singlebox SL_n(Z[S^{-1}])\backslash \GL_n(Q)/\GL_n(Z_S)\cdot \{\lambda I_n| \lambda \in Z[S^{-1}]\}=SL_n(Z[S^{-1}])\backslash Y_S(V)\esinglebox\]
\end{proof*}

\section{Spaces with actions of general linear groups}\label{sec:spacesWith}

\subsection{$\GL_n(\IZ)$ acts on the space of homothety classes of inner products}\label{sec:nfspace}
This section will analyze the metric on the space of homothety classes of inner products  (defined for example in \cite[p.~314~ff.]{bridson1999metric}). Furthermore certain properties of the volume functions will be established. Apart from the growth condition, which was analyzed in \cite[Section~1]{Bartels-Lueck-Reich-Rueping(2012KandL)}, these have basically been shown in \cite{Grayson(1984)}. It still makes sense to restate them in precisely this form. Then the localized version for $\IZ$ and for $F[t]$ can be treated simultaneously in Section~\ref{sec:volloc}.

Let $V$ be finitely generated, free $\IZ$-module of rank $n$ and consider the space $\tilde{X}(V)$ of all inner products on $\IR\otimes_\IZ V$. We will think of an inner product on $\IR\otimes_\IZ V$ either as a symmetric map $\IR\otimes_\IZ V\rightarrow (\IR\otimes_\IZ V)^*$ or as a bilinear form. 

After a choice of a $\IZ$-basis for $V\subset \IR\otimes_\IZ V$ we can write such an inner product as a matrix. This gives $\tilde{X}(V)$ the structure of a manifold. Rescaling gives a group action of $(\IR^{>0},\cdot)$ on $\tilde{X}(V)$ via
\[(\lambda, s)\mapsto \lambda s.\]
Let $X(V)$ be the quotient of $\tilde{X}(V)$ under this group action. An element of $X(V)$ is called a homothety class of inner products. The projection map has a section that sends a homothety class to the inner product whose representing matrix with respect to some basis of $V$ has determinant one. The group $\aut_\IZ(V)\cong \GL_n(\IZ)$ acts on the space of homothety classes of inner products.

$\tilde{X}(V)$ is a subset of the vector space $\Sym(\IR\otimes_\IZ V)$ of symmetric linear maps $(\IR\otimes_\IZ V)\rightarrow (\IR\otimes_\IZ V)^*$. Symmetric means that for any $s\in \Sym(V)$ the map 
\[(\IR\otimes_\IZ V)\stackrel{\cong}{\rightarrow} (\IR\otimes_\IZ V)^{**}\stackrel{s^*}{\rightarrow}(\IR\otimes_\IZ V)^*\]
is again $s$. The isomorphism on the left is the inverse of the canonical evaluation isomorphism. 
 
Indeed $\tilde{X}(V)$ is an open subset of $\Sym(\IR\otimes_\IZ V)$. So we get a canonical trivialization of the tangent bundle 
\[\tilde{X}(V)\times \Sym(\IR\otimes_\IZ V)\stackrel{\cong}{\rightarrow} T_*\tilde{X}(V) \qquad (s,v)\mapsto [t\mapsto s+tv].\]
Let us now define a Riemannian metric on $\tilde{X}(V)$. So we have to define for each $s\in X(V)$ an inner product $g_s$ on $\tilde{X}(V)$:
\[g_s(u,v)\coloneqq \tr(s^{-1}\circ u\circ s^{-1} \circ v).\]
It is obviously bilinear and symmetric. Furthermore the endomorphism $s^{-1}\circ u$  is self adjoint with respect to the inner product $s$ on $(\IR\otimes_\IZ V)$ since 
\[s^{-1}\circ (s^{-1}\circ u)^*\circ s = s^{-1} u.\]
Hence there is an orthonormal basis of eigenvectors with eigenvalues $\lambda_1,\ldots,\lambda_n$. Then the eigenvalues of $(s^{-1}\circ u)^2$ are $\lambda_1^2,\ldots, \lambda_n^2$ and its trace is just the sum. Hence its trace is nonnegative and it vanishes only if $s^{-1}u$ is zero. 
In this case $u$ is zero since $s$ is invertible. So $g_s$ is indeed an inner product.

\begin{lemma}\label{lem:nfVolLipschitz} Let $V$ be a free abelian group of rank $n $. The function $\ln\circ \vol_W:\tilde{X}(V)\rightarrow \IR$ is $n$-Lipschitz for any direct summand $W$ of $V$.
\end{lemma}
\begin{proof} The strategy is just to compute the gradient and observe, that its length is $\frac{1}{2}\sqrt{\rk(W)}\le n$. See for example \cite[Corollary~1.8]{Bartels-Lueck-Reich-Rueping(2012KandL)}.
\end{proof}

\begin{corollary}\label{cor:nfcWLipschitz} Thus $c_W$ is $4n$-Lipschitz in the above setting.
\end{corollary}

We still need one preliminary lemma.
\begin{lemma}\label{lem:thinningOpen} Let $X$ be a proper, inner metric space and let $U\subset X$ be an open subset and $\beta \in \IR$ be any real number. Then 
\[U^{-\beta}\coloneqq \{x\in U\mid\overline{B}_\beta(x)\subset U\}\]
is open.
\end{lemma}
\begin{proof}
We have to show that there is for $x\in U^{-\beta}$ an $\varepsilon'>0$ such that $B_{\varepsilon'}(x)\subset U^{-\beta}$.

Since $U$ is open there is for each $z\in \overline{B }_\beta(x)$ an $\varepsilon(z)\in\IR$ with $B_{\varepsilon(z)}(z)\subset U$. The set $\overline{B}_\beta(x)$ is compact as the metric space is proper. So there is a uniform $\varepsilon>0$ with $B_\varepsilon(z)\subset U$ for all $z\in \overline{B}_\beta(x)$. 

Since the metric space is inner we get
\[B_{\beta+\varepsilon}(x)=\bigcup_{z\in B_\beta(x)}B_\varepsilon(z)\]
and hence it is contained in $U$. Hence by the triangular inequality $B_\frac{\varepsilon}{2}(x)$ is contained in $U^{-\beta}$ and hence it is open.
\end{proof}

\begin{proposition}\label{prop:nf:coversAtInfinitySatisfied} The space $X(V)$ satisfies all assumptions from Proposition~\ref{prop:1:coversAtInfinity}. Let 
\[\calw\coloneqq \left\{\{x\in X(V)\mid c_W(x)> 0\}\mid W \subset V \mbox{ is a nontrivial direct summand}\right\}.\] This is a collection of open sets as the map $c_W:X(V)\rightarrow \IR$ is continuous. We have
\begin{enumerate}
\item \label{prop:nf:coversAtInfinitySatisfied:eins} $X(V)$ is a proper $\CAT(0)$ space, 
\item \label{prop:nf:coversAtInfinitySatisfied:zwei} the covering dimension of $X(V)$ is less or equal to $\frac{(n+1)n}{2}-1$,
\item \label{prop:nf:coversAtInfinitySatisfied:drei} the group action of $\aut_{\IZ}(V)\cong \GL_n(\IZ)$ on $X$ is proper and isometric,
\item \label{prop:nf:coversAtInfinitySatisfied:vier} $\aut_\IZ(V)\cdot \calw \coloneqq \{gW\mid g\in \aut_\IZ(V), W\in \calw\}=\calw$,
\item \label{prop:nf:coversAtInfinitySatisfied:fuenf} $gW$ and $W$ are either disjoint or equal for all $g\in \aut_\IZ(V), W\in \calw$,
\item \label{prop:nf:coversAtInfinitySatisfied:sechs} the dimension of $\calw$ is less or equal to $n-2$. 
\item \label{prop:nf:coversAtInfinitySatisfied:sieben} the $\aut_{\IZ}(V)$ operation on 
\[X\setminus(\bigcup \calw^{-\beta})\coloneqq \{x\in X\mid \nexists W\in \calw:\overline{B}_\beta(x)\subset W\}\]
is cocompact for every $\beta \ge 0$.
\end{enumerate}
\end{proposition}
\begin{proof}
\begin{enumerate}
\item See for example \cite[Chapter~II Theorem~10.39]{Bridson-Haefliger(1999)}.
\item After choosing a basis for $V$ we can identify the space $X(V)$ with the set of positive definite, symmetric $n\times n$ matrices of determinant one. This is a Riemannian manifold of dimension $\frac{(n+1)n}{2}-1$. Its covering dimension is at most $\frac{(n+1)n}{2}-1$ by \cite[Corollary~50.7]{Munkres(1975)}.
\item A straightforward computation shows that the group action is isometric. Since $X(V)$ embeds equivariantly in the space of all inner products on $\IR\otimes V$, we can consider this space instead. Pick for a point $s\in X(V)$ the compact set $K\coloneqq B_1(s)$. Fix a basis of $V$ and 
let $C$ denote the length of the longest element of this basis. If $gK\cap K\neq \emptyset$, then $gs$ and $s$ have distance at most $2$. Thus every element of the upper basis has length at most
$C\cdot e^{2n}$ by Lemma~\ref{lem:nfVolLipschitz}. So there are only  finitely many group elements with $gK\cap K\neq \emptyset$. Thus action is proper.
\item Pick an element $g\in \aut_\IZ(V)$ and an open set $U\in \calw$. It has the form $U=\{x\in X(V)\mid c_W(x)> 0\}$ for a nontrivial direct summand $W\subset V$. We have $c_W(s\cdot g)=c_{gW}(s)$ and hence $gU=\{x\in X(V)\mid c_{gW}(x)> 0\}\in \calw$.
\item Assume $x\in gU\cap U$ for some $U\in \calw,g\in \aut_\IZ(V)$. Thus $c_W(x)>0$ and $c_{gW}(x)> 0$. 
By Corollary~\ref{cor:cWAndCanFiltr} this means that $W,gW$ are both contained in the canonical filtration. And since they have the same rank they have to be equal.
\item Suppose $x\in \bigcap_{i=1}^m U_i$ for some $U_i\in \calw$. Then $U_i$ can be written as $\{x\in X(V)\mid c_{W_i}(x)> 0\}$ for some nontrivial direct summands $(W_i)_{i=1\ldots,m}$. Hence they all have to occur in the canonical filtration. The canonical filtration can have at most one module for each rank between one and $n-1$. Thus $m\le n-1$. So the dimension of $\calw$ is at most $n-2$.
\item Let us show that $X(V)\setminus(\bigcup \calw^{-\beta})$ is a closed subset of a cocompact set. We have already shown in $(iv)$ that it is $G$-invariant. By Lemma~\ref{lem:thinningOpen} it is a closed subset of $X(V)$. 
By Corollary~\ref{cor:nfcWLipschitz} we know that each function $c_W$ is $4n$-Lipschitz. Hence 
\begin{eqnarray*}
&& X\setminus(\bigcup \calw^{-\beta})\\
&\subset& \{x\in X(V)\mid c_W(x)\le 4n\beta \mbox { for each nontrivial direct summand } W \subset V\}.\end{eqnarray*}
The group operation on the right hand side is cocompact by \cite[Corollary~5.2]{Grayson(1984)}. Hence the group operation on the closed subset $X\setminus(\bigcup \calw^{-\beta})$ is also cocompact.
\end{enumerate}
\end{proof}
\subsection{Preliminaries about affine buildings}\label{ssect:prelimAffine}

Most of this subsection can be found in \cite{garrett1997buildings}. Basics about Euclidean simplicial complexes or more generally about $M_k$-polyhedral complexes can be found in \cite[Chapter~I.7]{Bridson-Haefliger(1999)}. 

Let us begin with some preliminaries about affine buildings. Let $O$ be a  discrete valuation ring with fractional field $k$. Let $m$ be the unique maximal ideal of $O$ and let $\kappa$ denote the residue field $O/m$. Let $t$ be a generator of $m$.
Let $V$ be an $n$-dimensional vector space over $k$.

A homothety is a $k$-linear map of the form 
\[V\rightarrow V \qquad v\mapsto \lambda v\]
for some $\lambda \in k\setminus \{0\}$. Two $O$-lattices $L_1,L_2\subset V$ are homothetic if there is a homothety $f:V \rightarrow V$ with $f(L_1)=L_2$. Being homothetic is an equivalence relation and we write $[L_1]$ for the homothety class of $L_1$.

Now we can consider a simplicial complex whose vertex set is the set of all homothety classes of $O$-lattices in $V$ and where a sequence $[L_1],\ldots,[L_m]$ of equivalence classes spans a simplex if there are representatives such that 
\[L_1\subset L_2\subset \ldots \subset L_m\subset t^{-1}L_n.\]

\begin{lemma}\label{lem:AffBuilLocFin} The set of neighbors of a vertex $[L]$ can be identified with the set of $\kappa$-subspaces of the $n$-dimensional $\kappa$-space $t^{-1}L/L$.

Especially if $\kappa$ is finite the complex $X(V)$ is locally finite. This condition is automatically satisfied for $k=\IQ$ or $k=F(t)$ for a finite field $F$.
\end{lemma}
\begin{proof}
By definition $m\cdot t^{-1}L=(t)\cdot t^{-1}L=L$ and hence $t^{-1}L/L$ has the structure of a $\kappa$-module. Any isomorphism $L\cong O^n$ induces $t^{-1}L/L\cong t^{-1}O^n/O^n\cong (O/tO)^n=\kappa^n$.

For two adjacent vertices $[L]$ and $[L']$ and a representative $L$ of $[L]$ we can find a unique representative $L'$ of $[L']$ such that $L\subset L'\subset t^{-1}L$. Assigning to it the $\kappa$-subspace $L'/L\subset t^{-1}L/L$ gives the desired bijection.
\end{proof}

\begin{definition}\label{def:labelDifference}
We can furthermore label the vertices with elements in $\IZ/n$. Let us first pick a base vertex $[L]$ with a representative $L$. Since $\bigcup_{n\in \IN} t^{-n} L'=V$ we find an $n$ such that $t^{-n}L'$ contains all generators of $L$. By changing the representative $L'$ we thus may assume that $L\subset L'$.

Define the label of $[L']$ to be $l([L'])\coloneqq \dim_\kappa(L'/L) \mod n$. We can check that this labeling does not depend on the choice of representatives. Furthermore it can also be expressed as the valuation of the determinant of a base change matrix from an $O$-basis of $L$ to an $O$-basis of $L'$.

The difference between the labeling $([L''],[L'])\mapsto l([L'])-l([L''])$ is even independent of the choice of the base vertex. It can be expressed as $\dim_\kappa(L'/L'') \mod n$ where $L',L''$ are representatives of $[L'],[L'']$ with $L''\subset L'$.

For an edge $e$ with endpoints $[L],[L']$ let the label difference of $e$ denote $\pm (l([L'])-l([L'']))$ in the set $(\IZ/n)/ x\sim -x$.
\end{definition}

\begin{lemma} \label{lem:simplicialactionsPreserveLabelDiffs}An edge $e$ of label difference $k$ is contained in 
\[\prod_{i=1}^{k} \frac{r^i-1}{r-1}\cdot \prod_{i=1}^{n-k} \frac{r^i-1}{r-1} \]
$n-1$-dimensional simplices. The number $r$ denotes the cardinality of $\kappa$. Especially the label differences of two edges with isomorphic links are equal.
\end{lemma}
\begin{proof}
Let $L,L'$ be representatives of the endpoints of the $e$ with $L\subset L'\subset t^{-1}L$. Then $\dim_\kappa(L/L')\in \{k,n-k\}$. So we can assume that this dimension is $k$. Otherwise replace $L$ by $L'$ and $L'$ by $t^{-1}L$. Each such $(n-1)$-dimensional simplex then corresponds to a flag of the form
\[L\subset L_1 \subset \ldots \subset L_{k-1} \subset L' \subset L_{k+1}\subset \ldots L_{n-1}\subset t^{-1} L.\]
By dividing $L$ out each such flag corresponds to a flag 
\[0\subset V_1 \subset \ldots \subset V_{k-1} \subset L'/L \subset V_{k+1}\subset \ldots V_{n-1}\subset (t^{-1} L)/L.\]
of the $n$-dimensional $\kappa$-vector space $(t^{-1} L)/L$ containing $V_k\coloneqq L'/L$. 
Assume we already picked $V_i$ and we want to pick $V_{i+1}$ for $i+1<k$. So we have to pick a vector $v_{i+1} \in V_k$ that does not lie in $V_i$. There are $p^k-p^i$ choices for such a vector. And two vector yield the same vector space $V_{i+1}:=\langle v_{i+1},V_i\rangle$ if they differ multiplicatively by a unit in $\kappa$ and additively by some element of $V_i$. So there are $\frac{r^k-r^i}{r^i(r-1)}=\frac{r^{k-i}-1}{r-1}$ such choices possible. The analogous argument holds for $i\ge k$ and yields 
\[\prod_{i=0}^{k-1} \frac{r^{k-i}-1}{r-1}\cdot \prod_{i=0}^{n-k-1} \frac{r^{n-k-i}-1}{r-1}. \]
A final substitution yields the desired result.
Now assume that $k\le n/2$. Thus $k\le n-k$. Let 
\[f(k)\coloneqq \prod_{i=0}^{k-1} \frac{r^{k-i}-1}{r-1}\cdot \prod_{i=0}^{n-k-1} \frac{r^{n-k-i}-1}{r-1}.\]
We have 
\[\frac{f(k-1)}{f(k)}=\frac{r^{n-k+1}-1}{r^k-i}>1.\]
Thus $f$ is monotonically decreasing on $1,\ldots, \lfloor \frac{n}{2}\rfloor$. This is a complete system of representatives of $(\IZ/n)/x\sim-x$. 
So the induced map $(\IZ/n)/x\sim-x\rightarrow \IN$ is injective. This proves the last claim.
\end{proof}

A Euclidean $n$-simplex is the convex hull of $n+1$ points in $\IR^n$ in general position. An Euclidean simplicial complex is a simplicial complex where any simplex carries additionally the structure of an Euclidean simplex. This means that we can identify the vertices of the simplex with the vertices of the given Euclidean simplex. Furthermore the inclusions of the faces are required to be isometries. See \cite[Chapter~I, Definition~7.2]{Bridson-Haefliger(1999)} for the precise definition.

Let us recall the definition of a building as given in \cite[Chapter~I Definition~10A.1]{Bridson-Haefliger(1999)}. It is not the usual definition of an affine building; for example it already requires a metric.
\begin{definition} A Euclidean building of dimension $n-1$ is a piecewise Euclidean simplicial complex $X$ such that:
\begin{enumerate}
\item $X$ is the union of a collection $\cala$ of subcomplexes $E$, called apartments, such that the intrinsic metric $d_E$ on $E$ makes $(E, d_E)$ isometric to the Euclidean space $E^n$ and induces the given Euclidean metric on each simplex. The $n-1$-simplices of $E$ are called its chambers.
\item Any two simplices $B$ and $B'$ of $X$ are contained in at least one apartment.
\item Given two apartments $E$ and $E'$ containing both the simplices $B$ and $B'$, there is a simplicial isometry from $(E, d_E)$ onto $(E' , d_{E'})$ which leaves both $B$ and $B'$ pointwise fixed.
   
   The building $X$ is called thick if the following extra condition is satisfied:
\item Thickness Condition: Any $(n-2)$-simplex is a face of at least three $n-1$-simplices.
\end{enumerate}
\end{definition}

Up to now the affine building is just a simplicial complex. We can furthermore equip the simplicial complex with the structure of an Euclidean simplicial complex. But first we need a preliminary lemma:
\begin{lemma} \label{lem:Rnsimplicial} Every $x\in \IR^n$ can be written uniquely as a convex combination $x=\sum_{i=0}^m \mu_i p_i$ with 
$p_i\in \IZ^n,0<\mu_i\le 1,\sum_{i=0}^m \mu_i=1$ such that 
 $p_0< \ldots< p_m\le p_0+(1,\ldots, 1)$. Here $a\le b$ means $a_i\le b_i$ for all $i$. Especially this implies $m\le n$. 
\end{lemma}
\begin{proof}

This triangulation of $\IR^n$ is obtained from the tesselation with cubes by a certain subdivision into simplices.

Let me just give a sketch of the proof.
The statement is trivial for $n=0$. So let $x\in \IR^n$ be given. Without loss of generality we can assume that $\lfloor x_i\rfloor =0$ for all $i$. By permuting the coordinates we can assume that
\[1> x_1\ge \ldots \ge x_n\ge 0 .\]
Now let $m$ be the number of different entries of $x$ and let $(x'_1,\ldots, x'_m)$ be obtained from $x$ by leaving out coordinates that occur twice. Let $\chi_{> x'_i}$ be the characteristic function \[y\mapsto \begin{cases} 1 & y> x_i\\0 &\mbox{else} \end{cases}\] 
and let $p_i\coloneqq \chi_{>x'_i}$.
Then $x$ can be written as a convex combination of the $p_i$. Conversely if you know that $x$ can be written as a convex combination of a totally ordered subset of $\{0,1\}^n$ with nonzero coefficients you can read off that subset by comparing the coordinates of $x$.
\end{proof}

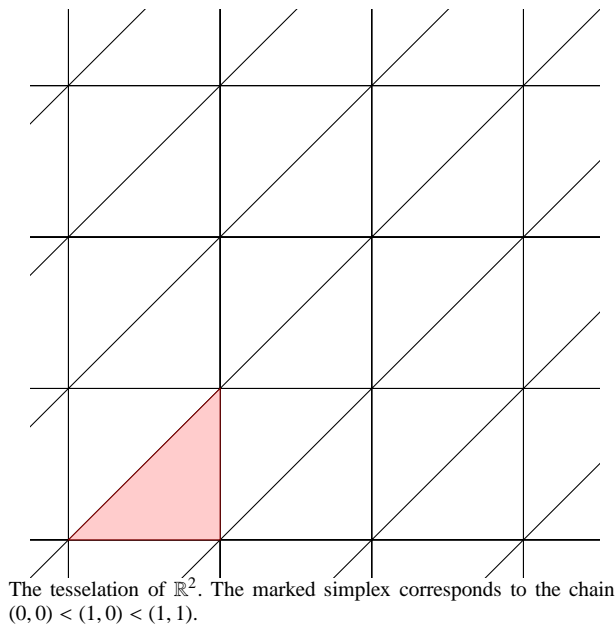
\begin{figure}[ht]
  \centering
  \begin{tikzpicture}
    \coordinate (Origin)   at (0,0);
    \coordinate (XAxisMin) at (-0.4,0);
    \coordinate (XAxisMax) at (8,0);
    \coordinate (YAxisMin) at (0,-5);
    \coordinate (YAxisMax) at (0,2);
       \clip (-0.5,-0.5) rectangle (7cm,7cm);  
   \foreach \x in {-2,0,...,6}{
		\foreach \y in {-2,0,...,6}{
			\draw  (\x,\y)--+(2,0)--+(0,0)--+(0,2)--+(0,0)--+(2,2);
		}
   }
     \filldraw[fill=red!20,draw=red!50!black]
(0,0)--+(2,0)--+(2,2)--cycle;


  \end{tikzpicture}
  \vspace*{8pt}
   \begin{minipage}{8cm}
The tesselation of $\IR^2$. The marked simplex corresponds to the chain $(0,0)<(1,0)<(1,1)$.
   \end{minipage}   
  \label{fig:tesR2}
\end{figure}

\begin{remark}\label{rem:DiagSimplices}
The convex combination for $x+\lambda (1,\ldots,1)$ can be obtained from the convex combination for $x$ in the following way. Let us assume without loss of generality that $\lambda$ is positive; otherwise swap the roles. 
Make the coefficient of $p_0$ smaller and increase the coefficient of $p_0+(1,\ldots,1)$ correspondingly until the coefficient of $p_0$ becomes zero. Then $p_1$ is the smallest element from $\IZ^n$ needed and we can continue this way: Now decrease the coefficient of $p_1$ and increase the coefficient of $p_1+(1,\ldots,1)$. 
\end{remark}

Now we are ready to define the metric on the affine building: 

\begin{proposition} \label{prop:affbuildingCAT0}The affine building $X$ has the following properties:
\begin{enumerate}
\item For each basis $b_1,\ldots,b_n$ of $V$ we can consider the full subcomplex $X'$ spanned by all vertices of the form $[t^{m_1}b_1,\ldots,t^{m_n}b_n]$ for $m\in \IZ^n$. This will be an apartment of the building.

We can map such a vertex to $\pr(m_1,\ldots,m_n)\in \IR^n$, where $\pr:\IR^n \rightarrow \langle (1,\ldots,1)\rangle^\perp$ denotes the orthogonal projection with respect to the standard inner product on $\IR^n$. The linear extension $f:X'\rightarrow \{x\in \IR^n\mid \sum_{i=1}^n x_i=0\}$ of this map is a bijection. We can pull the metric on $\{x\in \IR^n\mid \sum_{i=1}^n x_i=0\}$ 
back to each simplex to obtain an Euclidean simplicial complex.
\item \label{prop:affbuildingCAT01}The length of an edge in $X'$ depends only on the label difference of its endpoints. 
\item \label{prop:affbuildingCAT02}If a simplex is contained in two apartments we get the same metric on that simplex.
\item \label{prop:affbuildingCAT03}A simplicial automorphism $g:X\rightarrow X$ is an isometry.
\item \label{prop:affbuildingCAT04}$\aut_k(V)$ acts isometrically on $X$. 
\item \label{prop:affbuildingCAT05}Any two simplices are contained in at least one apartment.
\item \label{prop:affbuildingCAT06}Given two apartments $E$ and $E'$ containing both the simplices $B$ and $B'$, there is a simplicial isometry from $(X, d_X)$ onto $(X' , d_{X'})$ which leaves both $B$ and $B'$ pointwise fixed.
\item \label{prop:affbuildingCAT07} The affine building $X$ is a $\CAT(0)$ space.
\end{enumerate}
\end{proposition}
\begin{proof}
\begin{enumerate}
\item We have to show that each point $p\in \{x\in \IR^n\mid \sum_{i=1}^n x_i=0\}$ lies in the image of a unique open simplex. Let us first we can apply Lemma~\ref{lem:Rnsimplicial} to write it as a convex combination of certain points $p_1,\ldots, p_m$ of $\IZ^n$:
$\sum_{i=1}^m \mu_i  \cdot p_i=p=\pr(p) = \sum_{i=1}^m \mu_i  \cdot \pr(p_i)$. Hence $p$ lies in the convex hull of the points $(f([t^{p_{i,1}}b_1,\ldots,t^{p_{i,n}}b_n]))_{i=1\ldots m}$. The conditions on $p_i$ from Lemma~\ref{lem:Rnsimplicial} mean exactly that the vertices $[t^{p_{i,1}}b_1,\ldots,t^{p_{i,n}}b_n]$ span a simplex. Uniqueness follows from Remark~\ref{rem:DiagSimplices}. So $f$ is really a continuous bijection. Since $f$ is proper it is a homeomorphism. 

The images vertices of each simplex are in general position since otherwise there would be a point that can be written as a convex combination of those vertices in two different ways which we have already ruled out. So one obtains the structure of an Euclidean simplicial complex.

The space $ \{x\in \IR^n\mid \sum_{i=1}^n x_i=0\}$ is a convex subset of $\IR^n$. Hence the restriction of the standard metric to it is inner. The metric on the realization on an Euclidean simplicial complex is the unique inner metric whose restriction to each simplex agrees with the metrics given on it. So the realization of $X'$ is really isometric to $\IR^n$.

\item Let $e\in X'$ be any edge. Pick representatives 
\[\langle t^{m_1}b_1,\ldots,t^{m_n}b_n\rangle \mbox{ and }\langle t^{m'_1}b_1,\ldots,t^{m'_n}b_n\rangle\]
 of its endpoints $p,p'$ with 
\[\langle t^{m_1}b_1,\ldots,t^{m_n}b_n\rangle \subset \langle t^{m'_1}b_1,\ldots,t^{m'_n}b_n\rangle\subset t^{-1}\langle t^{m_1}b_1,\ldots,t^{m_n}b_n\rangle.\]
This means exactly that $m'_i$ is either $m_i-1$ or $m_i$. Note that 
\[\dim_\kappa( \langle t^{m'_1}b_1,\ldots,t^{m'_n}b_n\rangle/ \langle t^{m_1}b_1,\ldots,t^{m_n}b_n\rangle)= \sum_{i=1}^nm'_i-m_i. \]
Now we can consider the distance between $f(p)$ and $f(p')$. It is $||\pr(m_i'-m_i)||$. The length of a vector whose entries are either zero or one depends only on the number $r$ of ones. Not all entries can be simultaneously zero (or one) since then the endpoints of $e$ would be the same. This is impossible in a simplicial complex. 
 But we now the residue $r$ mod $n$ is just the label difference of the vertices. Since the desired number must be at least one and can be at most $n-1$ this determines $r$. So length of an edge depends only on its label difference. 
\item The metric on an Euclidean simplex is uniquely determined by the length of its edges. As shown before the length of an edge depends only on the label difference and not on the choice of some apartment.
\item Each simplicial automorphism of $X$ preserves the label difference by Lemma~\ref{lem:simplicialactionsPreserveLabelDiffs}. Thus it is an isometry.
\item $\varphi \in \aut_k(V)$ preserves the label difference since for two lattices $L,L'$ with $L\subset L'\subset t^{-1}L$ we have 
$\varphi(L)\subset \varphi(L')\subset t^{-1}\varphi(L)$ and $\varphi(L'/L)\cong L'/L$.

\item The proof can be found in \cite[chapter 19, p. 289]{garrett1997buildings}.

\item The proof goes as in \cite[chapter 19, p. 290]{garrett1997buildings}. The automorphism constructed there is simplicial. Hence it is an isometry from one apartment to the other by the same argument as above. 

\item \cite[Chapter~I Theorem~10A.4(ii)]{Bridson-Haefliger(1999)}.

\end{enumerate}
\end{proof}



We need the following lemma to deal with the properness of the affine building. 

\begin{lemma}[{\cite[Chapter I.3 Corollary 3.8]{bridson1999metric}}]\label{lem:locComCompleteInnerImplProper} A inner metric space is proper, if and only if it is locally compact and complete.
\end{lemma}

\begin{corollary}\label{cor:affBuildingProper}
The affine building $X$ is a proper metric space if the local field $\kappa$ is finite.
\end{corollary}
\begin{proof}
We want to use Lemma~\ref{lem:locComCompleteInnerImplProper}. The metric on the affine building is defined to be the inner metric induced by the Euclidean structure on the simplices. If $\kappa$ is finite, the simplicial complex is locally finite and hence locally compact. 
Furthermore the metric space is complete as mentioned above and shown in \cite[Chapter~I Theorem~7.13]{Bridson-Haefliger(1999)}.
\end{proof}

Furthermore we need another property of Euclidean simplicial complexes. 
\begin{proposition}\label{prop:extensionLipschitz} Let $X$ be an Euclidean simplicial complex with finitely many isometry types of simplices. Fix any $C\in \IR$. Then there is a $C'\in \IR$ such that the linear extension $\overline{f}$ of any function $f:X^{(0)}\rightarrow \IR$ with the property that $|f(x)-f(y)|\le C$ for any two adjacent vertices $x,y\in X^{(0)}$ is $C'$-Lipschitz.
\end{proposition}
\begin{proof}
We want to find an upper bound for 
\[\sup_{x,y\in X}\frac{f(x)-f(y)}{d(x,y)}.\] By definition of the metric it suffices to consider the case where $x,y$ lie in a common closed simplex. Let us construct a bound for each simplex. Since there are only finitely many isometry types of simplices, we can take their maximum.
So it suffices to consider the case where $X$ consists of only one Euclidean $n$-simplex.  

In this case $f$ is affine and the Lipschitz bound is the length of its gradient, which depends continuously on the given values at the vertices. Without loss of generality we can assign $0$ to one vertex. Thus the values at the other vertices are in the compact set $[-C,C]^n$. Since a continuous function on a compact has has a maximum, we obtain the desired result.
\end{proof}

\subsection{$\GL_n(F[t])$ acts on a building}

Let $F$ be a finite field and let $V$ be an $n$-dimensional free $F[t]$-module. The group $\aut_{F[t]}(V)\cong \GL_n(F[t])$ acts on the affine building $X(V)$ associated to the valuation $\nu$ with 
\[\nu(\frac{f}{g})\coloneqq \deg(g)-\deg(f) \qquad \mbox{ for }\frac{f}{g}\in Q\coloneqq F(t).\]
It is a simplicial complex whose vertex set consists of all homothety classes of $R$-lattices in $Q^n$ where $R$ denotes the valuation ring with respect to this valuation.  
A generator for the maximal ideal in $R$ of $\nu$ is given by $\frac{1}{t}$. A subset $[S_1],\ldots,[S_m]$ spans a simplex if and only if there are representatives with $S_1\subset S_2\subset \ldots \subset S_m\subset tS_1$.

The goal of this section is to show that this space satisfies all assumptions from Proposition~\ref{prop:1:coversAtInfinity}.

\begin{lemma}\label{lem:FFieldproperAction} The affine building $X(V)$ has the following properties
\begin{enumerate} 
\item The group action of $\aut_Q(Q\otimes_{F[t]}V)$ is simplicial.
\item The group action of $\aut_{F[t]}(V)$ on $X(V)$ is proper.
\end{enumerate}
\end{lemma}
\begin{proof*}
The natural $\aut_Q(Q\otimes_{F[t]}V)$-action on the vertex set preserves the higher simplices.
 It suffices to show that the vertices have finite stabilizers. 
 The affine building is locally finite and connected and hence all stabilizer groups of vertices are commensurable. So it suffices to consider only one vertex. To simplify notation let us assume that the rank $n$ free $F[t]$ module $V$ is $F[t]^n$. The stabilizer of the vertex $[R^n]$ under the $\GL_n(F[t])$- action is the finite group
 \[\singlebox\GL_n(F[t])\cap (\GL_n(R)\cdot \{t^k|k\in \IZ\})=\GL_n(F).\esinglebox\]
\end{proof*}

\begin{remark} Nevertheless there is no bound on the order of the stabilizers. Consider the $R$-lattice $S\coloneqq \langle (1,0),(0,t^m)\rangle\subset Q^2 =Q\otimes_{F[t]} F[t]^2$ for some  $m\in \IZ$. We see that $\begin{pmatrix}1&x\\0&1\end{pmatrix}$ stabilizes $S$ whenever $\deg(x)\le m$. So the stabilizers get arbitrarily large if we choose $m$ bigger and bigger. Especially this also shows that the group action of $\aut_{F[t]}(V)$ on $X(V)$ is not cocompact for $V\cong F[t]^2$.
\end{remark}

We want to use the volume function from section~\ref{sec:FFvol} resp. the function $c_W$ from section~\ref{sec:canFilt} to construct certain open subsets. We can associate to any $R$-lattice $S\subset Q\otimes_{F[t]}V$ and any submodule $W\subset V$ a real number $c_W(S)$. By homothety invariance (Remark~\ref{rem:ff:homothetyInvOfcW}) this function descends to a function from the vertices of $X(V)$ to the real numbers. We can extend it linearly to get a function from the whole of $X(V)$ to $\IR$ which is also called $c_W$.

\begin{lemma}\label{lem:ffcWLipschitz} For any two adjacent vertices $x,x'\in X(V)$ we have 
\[|c_W(x)-c_W(x')|\le 4n .\]
Furthermore there is a number $C\in\IR$ such that the function $c_W:X(V)\rightarrow \IR$ is $C$-Lipschitz for all nontrivial direct summands $W$. 
\end{lemma}
\begin{proof}
Let $S,S'$ be representatives of the homothety classes of $x,x'$ with $B\subset B'\subset tB$. Then we have by Corollary~\ref{cor:FFieldvolumeOfNeighbors}
\[\log\vol_W(B) \le \log\vol_W(B') \le \rk_R(W)+\log\vol_W(B)\]
for any direct summand $W\subset V$. Inserting this in the definition of $c_W$ gives 
\[ |c_W(B) -c_W(B')| \le 4n\]
for any nontrivial direct summand $W\subset V$. So Proposition~\ref{prop:extensionLipschitz} gives the desired result.
\end{proof}

\begin{proposition}\label{prop:ff:coversAtInfinitySatisfied} The affine building $X(V)$ satisfies all assumptions from~\ref{prop:1:coversAtInfinity}. Let
\[\calw\coloneqq \{\{x\in X(V)\mid c_W(x)> 4n\}\mid W \subset V \mbox{ is a nontrivial direct summand}\}.\]
This is a collection of open sets as the map $c_W:X(V)\rightarrow \IR$ is continuous. We have
\begin{enumerate}
\item $X(V)$ is a proper $\CAT(0)$ space, 
\item the covering dimension of $X(V)$ is less or equal to $n-2$,
\item the group action of $\aut_{F[t]}(V)\cong \GL_n(F[t])$ on $X$ is proper and isometric,
\item $G\calw \coloneqq \{gW\mid g\in G, W\in \calw\}=\calw$.
\item Let $W,W'$ be two different submodules of the same rank. Then the open sets $c_W^{-1}([4n,\infty))$ and $c_{W'}^{-1}([4n,\infty))$ do not intersect. Especially
\[ gW\cap W \neq \emptyset \Rightarrow gW=W \]
for all $g\in \aut_{F[t]}(V),\; W\in \calw$.
\item The dimension of $\calw$ is less or equal to $n-1$. 
\item The $\aut_{F[t]}(V)$ operation on 
\[X\setminus(\bigcup \calw^{-\beta})\coloneqq \{x\in X\mid \nexists W\in \calw:\overline{B}_\beta(x)\subset W\}\]
is cocompact for every $\beta \ge 0$.
\end{enumerate}
\end{proposition}
\begin{proof}
\begin{enumerate}
\item It is a $\CAT(0)$ space by Proposition~\ref{prop:affbuildingCAT0}. Since the residue field $R/t^{-1}R\cong F$ is finite, it is also proper by Corollary~\ref{cor:affBuildingProper} . 
\item $X(V)$ is a simplicial complex of dimension $n-1$. Hence its covering dimension is also $n-1$ by \cite[Corollary~7.3]{Nagami1962}.
\item It is proper and simplicial by Lemma~\ref{lem:FFieldproperAction}. It furthermore preserves the label difference since for any two vertices $[L],[L']$ with representatives $L,L'$ such that $L'\subset L$ and any $\varphi \in \aut_{F[t]}(V)$ we have 
\[L/L'\cong \varphi(L)/\varphi(L').\]
 (compare Definition~\ref{def:labelDifference}). In Proposition~\ref{prop:affbuildingCAT0} we have shown that any label difference preserving simplicial automorphism of the building is an isometry.
\item This follows directly from $\log\vol_{g-}(g-)=\log\vol_-(-)$ and the definition of $c_W$ (Definition~\ref{def:cW}).
\item So let $g\in \aut_{F[t]}(V),U=\{x\in X(V)\mid c_W(x)>4n\}\in \calw$ for a nontrivial direct summand $W$ of $V$ be given. Let $x\in X(V)$ be given with $c_{W}(x)> 4n$ and $c_{W'}(x)> 4n$. The point $x$ is contained in a simplex $s$. The value of $c_W$ at $x$ is a convex combination of the values of $c_W$ at the vertices of $s$. By Lemma~\ref{lem:ffcWLipschitz} we have that $c_W([L
])$ and $c_{W'}([L])$ are positive for all vertices $[L]$ of $
s$ . Thus $W$ and $W'$ occur in the canonical filtration for $(V,L)$. So $c_W([L])$ and $c_{W'}([L])$ cannot be both larger than zero as explained in Lemma~\ref{lem:cWneg}. This lemma applies, since we have already verified in Proposition~\ref{prop:fflatticeForFiltr}, that the volume and the rank satisfy Convention~\ref{conv:latticeForFiltr}.

The second statement follows if we pick $gW$ as $W'$.
\item We have already seen in the previous item that $x$ cannot be contained in two sets from $\calw$  corresponding to modules of the same rank. Thus any point can be an element of at most $|\{1,\ldots,n-1\}|=n-1$ sets and hence the covering dimension is at most $n-2$.
\item Let $\beta>0$ be given. There is a constant $C>0$ such that the function $c_W$ is $C$-Lipschitz for every nontrivial direct summand $W$ by Lemma~\ref{lem:ffcWLipschitz}. Now let us pick a set $U=\{x\in X(V)\mid c_W(x)> 4n\}\in \calw$. Note first that $U^{-\beta}$ is open by Lemma~\ref{lem:thinningOpen}.

Consider the set $U'=\{x\in X(V)\mid c_W(x)> 4n+C\cdot\beta\}$. Since $c_W$ is $C$-Lipschitz the closed ball of radius $\beta$ around each $x\in U'$ is entirely contained in $U$ and thus $U'\subset U^{-\beta}$. So let us consider the system 
\[\calw'\coloneqq \{\{x\in X(V)\mid c_W(x)> 4n+C\beta\}\mid W \subset V \mbox{ is a nontrivial direct summand}\} \]
first. We have already shown that $\bigcup \calw'\subset \bigcup \calw^{-\beta}$ and hence
$X\setminus(\bigcup \calw^{-\beta})$ is a closed subset of $X\setminus(\bigcup \calw')$. Thus it suffices to show that the group action on $X\setminus(\bigcup \calw')$ is cocompact.

It suffices to show that there are finitely many $\aut_{F[t]}(V)$ orbits of vertices such that each $x\in X\setminus(\bigcup \calw^{-\beta})$ lies in a simplex with at least one vertex from the given finite set of vertices. \ignore{We could pick a finite set of representatives for those orbits and consider the union of the closed stars around each of the points. Each of them is a finite simplicial complex since $X(V)$ is locally finite by Lemma~\ref{lem:AffBuilLocFin}. So we have found a compact set whose translates cover $X\setminus(\bigcup \calw')$. So let us find the finite set of orbits.}
Let $x\in X\setminus(\bigcup \calw')$ be given. Thus each vertex $v$ of the simplex which contains $x$ satisfies $c_W(v)\le 8n+C\beta\eqqcolon C'$ for each nontrivial direct summand $W\subset V$. Otherwise if one of them was bigger than they all would be bigger than $4n+C\beta$ by Lemma~\ref{lem:ffcWLipschitz}. But $c_W(x)$ is defined to be a convex combination of those values. So it would also be bigger than $4n+C\beta$ which contradicts the choice of $x$.

So let $L$ be an $R$-lattice representing one of the vertices. By rescaling with a suitable power of $t$ we can assume that $\log\vol_V(L)\in [0,n-1]$. 

We can use Proposition~\ref{prop:ffieldIsomandCanFilt}. We make use of the numbers $r_i$ occurring there. Recall that they have the property that 
\[\sum_{i=1}^n r_i =\log\vol_V(L) \in [0,n-1],\]
\begin{equation}0\le r_{m+1}-r_m =c_{\langle w_1,\ldots,w_m\rangle}([L]) \le C'\label{eq:growthbound}.\end{equation}
This gives a bound on the size of the numbers $r_i$ by the following consideration. At least one of $r_i$ is $\ge 0$ since their sum is nonnegative. At least one of the numbers $r_1\ldots, r_n$ is less than $1$ since their sum is smaller than $n-1$. Since the numbers $r_1,\ldots,r_n$ are monotonically increasing there is an index $j$ such that $r_j\le 0$ and $r_{j+1}\ge 0$. Using the bound on the growths \eqref{eq:growthbound} we get
\[|r_i-r_j| \le C'\cdot |i-j|\qquad\mbox{ and }\qquad r_j \in [-C',0].\]
Hence each $r_i$ lies in $[-C'-n\cdot C', n\cdot C']$. This means that there are only finitely many isomorphism types of such $R$-lattices possible that could occur as $L$ by Proposition~\ref{prop:ffieldIsomandCanFilt}. And an isomorphism $(V,L)\cong (V,L')$  is just an element $g\in \aut_{F[t]}(V)$ with $\id\otimes g (L)=L'$.
So we have found the desired finite set of orbits. This completes the proof.
\end{enumerate}
\end{proof}

\subsection{$\GL_n(Z[T^{-1}])$ acts on a product of CAT(0)-spaces}

\begin{convention} 
Let \begin{itemize}
\item $Z$ denote either $\IZ$ or $F[t]$ for a finite field $F$ and let $Q$ denote its quotient field,
\item $T$ be a finite set of primes in $Z$,
\item $V$ be a free $Z[T^{-1}]$ module of rank $n$,
\item $X(V)$ be the space of homothety classes of inner products on $V$ (as in section~\ref{sec:nfspace}) in the case of $Z=\IZ$ respectively the affine building for the valuation $\nu(\frac{f}{g})=\deg(g)-\deg(f)$ on $Q$ in the case of $Z=F[t]$,
\item $Y_T(V)$ denote the product of the affine buildings of $V$ for each $p$-adic valuation $\nu_p$ on $Q$ with $p\in T$ metrized as a product of CAT(0)-spaces, 
\item $\tilde{Y}_T(V)$ denote the set of all integral structures on $V$ with respect to $T$, i.e. the set of all finitely generated $Z_T$-submodules of $Q\otimes_{Z[T]}V$ of rank $n$,
\item $D$ be the space $\IR^T$ equipped with the $\aut_{Z[T^{-1}]}(V)$-action 
\[(f,(x_p)_{p\in T})\mapsto (\nu_p(\det(f))+x_p).\]
\end{itemize}
\end{convention}

We will show in this section that the space $X(V)\times Y_T(V) \times D$ satisfies all requirements of Proposition~\ref{prop:1:coversAtInfinity}. Let us first establish a connection between the vertices of $Y_S(V)$ and the set of integral structures on the $Z[T^{-1}]$-module $V$.
Let $\Verti(B)$ denote the vertex set of a simplicial complex $B$. 

\begin{proposition}\label{prop:locVertAreIntStruc} Let $S$ be a (finite) set of primes. Then 
\begin{enumerate}
\item the map
\[\Psi : \prod_{s\in S}\tilde{Y}_{\{s\}}(Z_s\otimes_{Z_S}V)\rightarrow \tilde{Y}_S(V)\qquad (B_s)_{s\in S}\mapsto \bigcap_{s\in S} B_s\]
is an isomorphism of $\aut_Q(Q\otimes_{Z[S^{-1}]}V)\cong \GL_n(Q)$-sets with inverse $B\mapsto (\langle B\rangle_{Z_s})_{s\in S}$.

\item It induces an isomorphism of $\aut_Q(Q\otimes_{Z[S^{-1}]}V)$-sets
\[\Verti(\prod_{s\in S}Y_{\{s\}}(Z_s\otimes_{Z_S}V))\rightarrow \Verti(Y_S(V))\qquad [B_s]_{s\in S}\mapsto [\bigcap_{s\in S} B_s]\]
with inverse $B\mapsto (\langle B\rangle_{Z_s})_{s\in S}$.
\end{enumerate}
\end{proposition}
\begin{proof}
\begin{enumerate}
\item This is an easy computation.
\item The vertices of the simplicial complexes in consideration are homothety classes of integral structures. A homothety is an element in the center of $\aut_Q(Q\otimes_{Z[S^{-1}]}V)\cong \{\lambda\cdot \id\mid\lambda \in Q^*\}$. We have to show that 
$B,B'\in \tilde{Y}_S(V)$ are homothetic if and only if $\langle B\rangle_{Z_s}$ and $\langle B'\rangle_{Z_s}$ are homothetic for each $s\in S$. This is also a straightforward computation.
\end{enumerate}
\end{proof}

\begin{lemma}\label{lem:locsautactsCocomOnY}
Let $V$ be a finitely generated free $Z[T^{-1}]$-module. The group action of
\[\saut_{Z[T^{-1}]}(V)\coloneqq \{\varphi\in \saut_{Z[T^{-1}]}(V) \mid\det(\varphi)=1\}\]
 on $\prod_{s\in S} |Y_s(V)|$ is cocompact. 
\end{lemma}
\begin{proof}
$\prod_{s\in S} |Y_s(V)|$ has the structure of a locally finite simplicial complex by Lemma~\ref{lem:AffBuilLocFin} equipped with a simplicial group action. Thus it suffices to show that the action on the vertex set is cofinite.

 The previous lemma identifies this set with the set of all homothety classes of integral structures $Y_S(V)$ and the $\saut_{Z[T^{-1}]}(V)$-action on $Y_S(V)$ is cofinite by Proposition~\ref{prop:locSLcofinite}. 
\end{proof}

We can consider for any nontrivial direct summand $W\subset V$ the function $c_W:X(V)\times Y_S(V)\rightarrow \IR$ that is defined in the following way. If $y=(y_s)_{s\in S}$ is a tuple of vertices we can pick representatives and use the bijection from Proposition~\ref{prop:locVertAreIntStruc} to obtain an integral structure $B$. Corollary~\ref{cor:loccWscalingInv} shows that $c_W(x,B)$ is independent of the chosen representatives. So we can assign to a point $(x,y)$ the value $c_W(x,B)$.

For general $y=(y_s)_{s\in S}$ we can write each $y_i$ as a convex combination of the vertices $v^s_1,\ldots,v^s_n$ of the open simplex in $Y_s(V)$ containing $y_s$, say $y_s = \sum_ {i=1}^n \lambda^s_i v^s_i$. Then define $c_W(x,y)$ as the linear extension in $y$-direction. More precisely
\[c_W(x,y)\coloneqq \sum_{i\in \{1,\ldots,n\}}(\prod_{s\in S}\lambda^s_{i_s})\cdot c_W(x,v^s_{i_s}).\]
Furthermore $Y_S(V)$ is a product of Euclidean simplicial complexes and thus it can be viewed as an Euclidean simplicial complex after a choice of simplex orientation that tells us how to subdivide the products of simplices.
\begin{lemma} \label{lem:loc:cWLipschitz}
\begin{enumerate}
\item \label{lem:loc:cWLipschitz1} Given $y\coloneqq (y_s)_{s\in S},y'\coloneqq (y'_s)_{s\in S}\in Y_S(V)$ such that each $y_s,y'_s$ is a vertex of $Y_s(V)$. 
Suppose that for each $s$ the vertices $y_s$  and $y'_s$ are either adjacent or equal. Then  we have
\[ |c_W(x,y)-c_W(x,y')| \le 4n\cdot \begin{cases}\ln(\prod_{s\in S}s)&Z=\IZ\\-\nu(\prod_{s\in S}s)&Z=F[t]\end{cases}.\]
\item There is a constant $C$ (independent of $W$ and $x$) such that $c_W(x,-)$ is $C$-Lipschitz.
\end{enumerate}
\end{lemma}
\begin{proof}
\begin{enumerate}
\item Pick for each $s\in S$ a $Z_s$-lattice $B_s$ in $Q\otimes_{Z[S^{-1}]}V$ representing $y_s$. As $y_s$ is adjacent or equal to $y'_s$ we can find a representative $B'_s$ of $y_s$ such that $sB_s\subseteq B'_s\subseteq B_s$. Now we have to consider the intersections $B\coloneqq \bigcap_{s\in S} B_s$ and $B'\coloneqq \bigcap_{s\in S} B'_s$. Let $z\coloneqq \prod_{s\in S}s$. We obtain since $B_s$ is a $Z_s$-module
\[zB=\bigcap_{s\in S} zB_s=\bigcap_{s\in S} sB_s\subset B'\subset B.\]
By definition $z$ is a product of elements from $S$. So we can use Corollary~\ref{cor:locVolOfAdjacentVert} to get 
\[|\log\vol_{W'}(x,B)-\log\vol_{W'}(x,B')|\le \rk(W')\cdot (-\nu(z)) \]
in the function field case; respectively 
\[|\log\vol_{W'}(x,B)-\log\vol_{W'}(x,B')|\le \rk(W')\cdot \ln(z)\]
in the number field case. If we insert this into the definition of $c_W$ we get
\[|c_{W}(x,B)-c_{W}(x,B')|\le 4n\cdot \begin{cases}\ln(z)&Z=\IZ\\-\nu(z)&z=F[t]\end{cases} .\]
\item $Y_S(V)$ is by definition a product of Euclidean simplicial complexes with finitely many isometry types of simplices. After subdividing products of simplices into simplices it inherits the structure of Euclidean simplicial complex with finitely many isometry types of simplices.
Note that vertices in the product can only be adjacent if they are adjacent or equal in each coordinate.

We have already computed a bound on the difference on two adjacent vertices in Lemma~\ref{lem:loc:cWLipschitz1}. Hence we can use Proposition~\ref{prop:extensionLipschitz} to conclude that there is a constant $C$ depending only on $n,S$ such that 
$c_W:X(V)\times Y_S(V)\rightarrow \IR$ is $C$-Lipschitz.
\end{enumerate}
\end{proof}

\begin{proposition}\label{prop:loc:coversAtInfinitySatisfied} The space $X(V)\times |Y_S(V)|\times D$ satisfies all assumptions from Proposition~\ref{prop:1:coversAtInfinity}. Let $R\in \IR$ be either $\ln(\prod_{s\in S}s)$ in the number field case or $\nu(\prod_{s\in S}s)$ in the function field case. 
Let $\calw\coloneqq \{\{(x,y,d)\in X(V)\times |Y_S(V)|\times D\mid c_W(x,y)> 4n(R+1)\}\mid W \subset V \mbox{ is a nontrivial direct summand}\}$. This is a collection of open sets as the map $c_W:X(V)\rightarrow \IR$ is continuous. We have
\begin{enumerate}
\item \label{prop:loc:coversAtInfinitySatisfied:eins} $X(V)\times |Y_S(V)|\times D$ is a proper $\CAT(0)$ space, 
\item \label{prop:loc:coversAtInfinitySatisfied:zwei}the covering dimension of $X(V)\times |Y_S(V)|\times D$ is less or equal to $\frac{n(n+1)}{2}-1+|S|n$,
\item \label{prop:loc:coversAtInfinitySatisfied:drei}the group action of $\aut_{Z[S^{-1}]}(V)\cong \GL_n(Z[S^{-1}])$ on $X(V)\times |Y_S(V)|\times D$ is proper and isometric,
\item \label{prop:loc:coversAtInfinitySatisfied:vier}$G\calw \coloneqq \{gW\mid g\in \aut_{Z[S^{-1}]}(V), W\in \calw\}=\calw$,
\item \label{prop:loc:coversAtInfinitySatisfied:fuenf}$gW\cap W \neq \emptyset \Rightarrow gW=W$ for all $g\in \aut_{Z[S^{-1}]}(V), W\in \calw$,
\item \label{prop:loc:coversAtInfinitySatisfied:sechs}the dimension of $\calw$ is less or equal to $n-2$,
\item \label{prop:loc:coversAtInfinitySatisfied:sieben}the $\aut_{Z[S^{-1}]}(V)$ operation on 
\begin{multline*}X(V)\times |Y_S(V)|\times D\setminus(\bigcup \calw^{-\beta})\\
=\{(x,y,d)\in X(V)\times |Y_S(V)|\times D| \nexists W\in \calw:\overline{B}_\beta(x)\subset W\}\end{multline*}
is cocompact for every $\beta \ge 0$.
\end{enumerate}
\end{proposition}
\begin{proof}
\begin{enumerate}
\item Each of the spaces $X(V)$ (Proposition~\ref{prop:nf:coversAtInfinitySatisfied}) and $|Y_S(V)|$ (Proposition~\ref{prop:affbuildingCAT0} for the $\CAT(0)$ condition and Corollary~\ref{cor:affBuildingProper} for the properness) and  $D$ ($\cong \IR^n$) is a proper $\CAT(0)$ space. Products of proper $\CAT(0)$ spaces are proper $\CAT(0)$ spaces (see for example \cite[Chapter~II Example~1.15(iii)]{Bridson-Haefliger(1999)} ).
\item All the spaces $X(V)$, $D$ and $Y_s(V)$ for $s\in S$ can be equipped with a CW-structure with countably many cells; in the number field case $X(V)$ is a smooth manifold and hence it can be triangulated. 

By \cite[Theorem~A.6]{hatcher2002algebraic} the product CW-structure on $X(V)\times |Y_S(V)|\times D$ really induces the product topology.  $X(V)$ is a $n(n+1)/2-1$-dimensional manifold in the number field case or a $n-1$-dimensional simplicial complex in the function field case. Each $Y_s(N)$ is a simplicial complex of dimension $n-1$. So $Y_S(V):=\prod_{s\in S} Y_s(V)$ is a CW-complex of dimension $|S|\cdot (n-1)$ and $D$ is a $|S|$-dimensional manifold. So the CW-dimension of $X(V)\times |Y_S(V)|\times D$ is at most $n(n+1)/2-1+|S|\cdot n$.
By \cite[Corollary~7.3]{Nagami1962} its covering dimension equals its dimension as a CW-complex.  So the covering dimension of the product is at most $n(n+1)/2-1+|S|\cdot n$.
\item  The group action on each factor is isometric by Proposition~\ref{prop:nf:coversAtInfinitySatisfied}\ref{prop:nf:coversAtInfinitySatisfied:drei} and Proposition~\ref{prop:affbuildingCAT0}\ref{prop:affbuildingCAT04}. So the action on the product is isometric. We have to show that it is a proper action (compare \cite[Chapter~I 8.2-8.3]{Bridson-Haefliger(1999)} for the subtilities of the definition of a proper action).

The last factors are realizations of simplicial complexes and the group action just permutes the vertices. Thus it suffices to show that every stabilizer group of a vertex $(y,d)$ in $|Y_S(V)|\times D|$ acts properly on $X(V)$.
Let $B$ be a free $Z_S$-module representing $y$. 
Given any $g\in\Stab([B])\cap \Stab(d)\subset \aut_{Z[S^{-1}]}(V)$. The condition $gd=d$ implies that $g$ fixes the integral structure $B$ on the nose and not only its homothety class. Thus $\Stab([B])\cap \Stab(d)$ normalizes $V\cap B\cong Z^n$. Thus it is the $Z[S^{-1}]$-linear extension of some $\aut_Z(V\cap B)$. 

So we just have to show that $\aut_Z(V\cap B)$ acts properly on $X(V)$. This has been done for the number field case in Proposition~\ref{prop:nf:coversAtInfinitySatisfied}\ref{prop:nf:coversAtInfinitySatisfied:drei} and for the function field case in Lemma~\ref{lem:FFieldproperAction}.
\item Let $g\in \aut_{Z[S^{-1}]}(V)$ and $W\in  \frl$ be given. We have
\begin{eqnarray*}
& & g\cdot \{(x,y,d)\in X(V)\mid c_W(x,y)> 4n(R+1)\}\\
&=&\{(x,y,d)\in X(V)\mid c_{W}(g^{-1}x,g^{-1}y)> 4n(R+1)\}\\
&=&\{(x,y,d)\in X(V)\mid c_{gW}(x,y)> 4n(R+1) \}.\end{eqnarray*}
The last equality uses Remark~\ref{rem:loc:equivarianceOfVol}\ref{rem:loc:equivarianceOfVol:drei}.
\item Let us proof first that two nontrivial direct summands $W,W'$ of $V$ of rank $m$ with 
\[c_W(x,y)>4n(R+1) \mbox{ and }c_{W'}(x,y)>4n(R+1)\]
for some point $(x,y)\in  X(V)\times Y_S(V)$ are equal. This will prove the statement since we have shown in the previous item that we get
\[gU=\{(x,y,d)\in X(V)\times |Y_S(V)|\times D\mid c_{gW}(x,y)>4n(R+1)\} .\]
for $U=\{(x,y,d)\in X(V)\times |Y_S(V)|\times D|c_W(x,y)>4n(R+1)\}$ and $g\in \aut_{Z[S^{-1}]}(V)$.

As mentioned above $Y_S(V)$ is an Euclidean simplicial complex. Let $s$ denote the open simplex containing $y$. The value of $c_W$ at $(x,y)$ is defined to be a convex combination of the values of $c_W(x,-)$ at the vertices of $s$. By Lemma~\ref{lem:loc:cWLipschitz} we see that all their values can differ at most by $4nR$. Thus the value at any vertex $v$ must be greater than $4n(R+1)-4nR=4n$. So $c_W(x,v)>4n,c_{gW}(x,v)>4n$. 

Fixing now the second coordinate we can consider the function $c_W(-,v):X(V)\rightarrow \IR$. Let $B$ be an representative of the homothety class of integral structures $y$. We have by Remark~\ref{rem:loccWIdent} $c_W(-,y)=c_{W\cap B}(-)$. So $c_{W\cap B}(x)>4n,c_{W'\cap B}(x)>4n$. By Proposition~\ref{prop:isomPoset} the two $Z$-submodules $W'\cap B,W\cap B$ of $V\cap B$ have the same rank. 
For the case of $Z=F[t]$ we can use Proposition~\ref{prop:ff:coversAtInfinitySatisfied} to conclude that $W'\cap B=W\cap B$.
For the number field case we can use Proposition~\ref{prop:nf:coversAtInfinitySatisfied} instead. Intersection with $B$ is an isomorphism of posets by Proposition~\ref{prop:isomPoset} and hence $W'$ and $W$ are equal.
\item Let $(x,y,d)$ be any point in $X(V)\times |Y_S(V)|\times D$. We have shown in the previous item that there can be at most one direct summand  $W$ for each rank between one and $n-1$ with $c_W(x,y)>4n(R+1)$. So there can be at most $n-1$ open sets in $\calw$ containing $(x,y,d)$. 
\item Of course, it suffices to show that $X(V)\times |Y_S(V)|\times D\setminus(\bigcup \calw^{-\beta})$ is a closed subset of a cocompact set. It is a  closed subset of the whole space by Lemma~\ref{lem:thinningOpen}. For any nontrivial direct summand $W$ the function $c_W$ is $C$-Lipschitz for a constant $C$ by Lemma~\ref{lem:loc:cWLipschitz}. Hence 
\begin{multline*} \{(x,y,d)\in X(V)\times |Y_S(V)|\times D\mid c_W(x,y)> 4n(R+1)+C\beta\}\\
\subset \{(x,y,d)\in X(V)\times |Y_S(V)|\times D\mid c_W(x,y)> 4n(R+1)\}^{-\beta}\end{multline*}
 and consequently 
$X(V)\times |Y_S(V)|\times D\setminus(\bigcup \calw^{-\beta})$ is a subset of 
\[ \{(x,y,d)\in X(V)\times |Y_S(V)|\times D \mid c_W(x,y)> 4n(R+1)+C \beta \}.\]
So we still have to show that the group action on the last set is cocompact. The group action of $\aut_{Z[S^{-}]}(V)$ on $D$ is cocompact. A fundamental domain is given by $K_D\coloneqq [0,1]^{|S|}$.

Consider the subgroup that stabilizes $K_D$ pointwise. It is 
\[H\coloneqq \{\varphi\in \aut_{Z[S^{-1}]}(V)\mid \det(\varphi)\in Z^*\}\mbox{ with }Z^*=\begin{cases}\{\pm 1\}&Z=\IZ\\ F^* & Z=F[t]\end{cases}.\]
It acts cocompactly on $Y_S(D)$ by Lemma~\ref{lem:locsautactsCocomOnY}. Thus there is a finite subcomplex $K_Y\subset Y_S(V)$ such that $H\cdot K_Y=Y_S(V)$. Let us consider the group that stabilizes every point in $K_Y$ pointwise. It is the intersection of the stabilizers of all vertices of $K_Y$ and thus it has finite index in the stabilizer of any vertex $v\in K_Y$ by Lemma~\ref{lem:StabVertCommensurable}. So consider 
\[\{g\in \saut_{Z[S^{-1}]}\mid gy=y\}.\]
As shown in Proposition~\ref{prop:loc:coversAtInfinitySatisfied:drei} before, this is just $\Stab(B)$ for any representative $B$ of the homothety class $y$. 
Again we have shown before that 
\[\{g\in \aut_Q(Q\otimes_{Z[S^{-1}]}V)\mid gV=V,gB=B\}=\{g\in \aut_Q(Q\otimes_{Z[S^{-1}]}V\mid gV\cap B=V\cap B\}.\]
The group on the right hand side is just $\aut_Z(V\cap B)$. 

So let us analyze the action of this group on $X(V)$. First note that we have by Proposition~\ref{prop:isomPoset}
\begin{eqnarray*}&&\{x\in X(V)\mid c_W(x,B)>4n(R+1)\mbox{ for a nontrivial direct summand } W\subset V\}\\
&=&\{x\in X(V)\mid c_{W'}(x)>4n(R+1)\mbox{ for a nontrivial direct summand } W'\subset V\cap B\}.\end{eqnarray*}
The group action on the complement of this set is cocompact. For the number field case this is shown in Proposition~\ref{prop:nf:coversAtInfinitySatisfied}\ref{prop:nf:coversAtInfinitySatisfied:sieben}. For the function field case this is shown in Proposition~\ref{prop:ff:coversAtInfinitySatisfied}.
\end{enumerate}
\end{proof}

\begin{lemma}\label{lem:StabVertCommensurable}
If a group $G$ acts simplicially on a locally finite simplicial complex $X$ the stabilizer groups of any two vertices are commensurable.
\end{lemma}
\begin{proof} Given any two vertices $x,y$ let $R$ denote the combinatorial distance between $x$ and $y$. As the simplicial complex is locally finite the set of all vertices of combinatorial distance $\le R$ to $x$ is finite and it contains $y$. Now the stabilizer group $G_x$ acts on this set. The isotropy group of $y$ under this restricted action is $G_x\cap G_y$. So we get an injection.
Hence the  index of $G_x\cap G_y$ in $G_x$ is finite. Analogously for $y$. Hence the subgroups $G_x$ and $G_y$ are commensurable.
\end{proof}

\section{Reducing the family}\label{sec:reducing}

Let $Z$ be either $\IZ$ or the polynomial ring over a finite field. Let $S$ be a finite set of primes in $Z$ and $F$ be any finite group. 
The term ``class of groups''  will always denote a class of groups that is closed under isomorphisms and taking subgroups. A family of subgroups of a group $G$ is a collection of subgroups, that is closed under taking subgroups and conjugation. A class of groups determines a family of subgroups; namely those which are in this class. Examples are the class of trivial groups, the class $\Fin$ of finite groups, the class $\VCyc$ of virtually cyclic groups and the class $\VSol$ of virtually solvable groups. For a family $\calf$ let $\calf_2$ denote the family of those groups containing a group from $\calf$ of index at most two.

\begin{notation}
Let us say that a triple $(\calh^?_*,G,\calf)$ \emph{satisfies the isomorphism conjecture} (in certain degrees), if the map 
\[\calh^G_*(E_\calf G)\rightarrow \calh^G_*(\pt)\]
is an isomorphism (in those degrees). 

If the isomorphism conjecture holds for $(H^?_*(-;{\bf K}_\cala),G,\VCyc)$ for any additive category $\cala$ then $G$ satisfies the \emph{K-theoretic \FJC} (Farrell-Jones conjecture). A group $G$ satisfies the \emph{L-theoretic \FJC}, if the isomorphism conjecture holds for $(H^?_*(-;{\bf L}^{\langle-\infty\rangle}_\cala),G,\VCyc)$ for any additive category $\cala$ with involution (compare \cite[Section~3 and Section~5]{Bartels-Reich(2007coeff)}).
A group $G$ satisfies the \emph{\FJC} if it satisfies both the K- and L-theoretic \FJC.

Let us say that a group $G$ satisfies the \emph{\FJC~relative to a family $\calf$} if we replace $\VCyc$ by $\calf$.
A group $G$ is said to satisfy the \emph{\FJC~with finite wreath products}, if the group $G\wr F$ satisfies the \FJC~for any finite group $F$.
\end{notation}

\begin{theorem}\label{thm:ICGLNF} Let $F$ be a finite group and let $\calf$ denote the family 
\[\VCyc\cup \{\Stab(W)\mid W\mbox{ is a nontrivial direct summand of }Z[S^{-1}]^n\}.\]
The group $\GL_n(Z[S^{-1}])\wr F$ satisfies the $K$- and $L$-theoretic \FJC~in all degrees with respect to the family $\calf^\wr$, which consists of those subgroups that have a finite index subgroup which is abstractly isomorphic to a finite product of groups from $\calf$. 
\end{theorem}
\begin{proof}
We have found a space satisfying the conditions from Proposition~\ref{prop:1:coversAtInfinity} (see Proposition~\ref{prop:nf:coversAtInfinitySatisfied} for the case of $\mathbb{Z}$,
Proposition~\ref{prop:ff:coversAtInfinitySatisfied} for the case of $F[t]$ and 
Proposition~\ref{prop:loc:coversAtInfinitySatisfied} for the localized versions).
Proposition~\ref{prop:1:coversAtInfinity} says that $\GL_n(Z[S^{-1}])\wr F$ satisfies the K- and L-theoretic \FJC~relative to the family $\calf^\wr.$ 
\end{proof}

The goal of this section is to reduce the family further as far as possible. We need the following two key properties:

\begin{theorem}[Transitivity principle~{\cite[Theorem~2.9]{Lueck-Reich(2005)}}] \label{prop:transitivityPrinciple} Let $\mathcal{H}^?_*$ be an equivariant homology theory and let $G$ be a group and let $\calf\subset \calf'$ be two families of subgroups. Suppose that each $H\in \calf'$ satisfies the isomorphism conjecture for the family $\calf|_H$. 

Then $G$ satisfies the isomorphism conjecture for the family $\calf$ if and only if it satisfies the isomorphism conjecture for the family $\calf'$.
\end{theorem}

\begin{proposition}[{\cite[Corollary~4.3]{Bartels-Reich(2007coeff)}}] \label{prop:IChomomrphisms} Let $f:G\rightarrow H$ be a group homomorphism. If $H$ satisfies the isomorphism conjecture (with finite wreath products) for a family $\calf$, then $G$ satisfies the isomorphism conjecture (with finite wreath products) for the family $f^*\calf$.
\end{proposition}

\begin{remark}\label{rem:ICsubgroup} If $G$ satisfies the isomorphism conjecture with respect to a family $\calf$, then each subgroup $H\le G$ satisfies the isomorphism conjecture with respect to the family $\calf|_H$.

So if $\calf$ is a subfamily of $\calf'$ and if a group $G$ satisfies the isomorphism conjecture with respect to $\calf$, then it also satisfies the isomorphism conjecture with respect to $\calf'$.
\end{remark}

\begin{theorem} \label{thm:FJCAT} The \FJC~with finite wreath products holds for any CAT(0)-group.
\end{theorem}
\begin{proof} Suppose $G$ acts properly, isometrically and cocompactly on a $\CAT(0)$ space $X$. Then $G\wr F$ acts the same way on $X^F$ and hence it is also a CAT(0) group. So it suffices to consider the version without wreath products.

This is then \cite[Theorem~B]{bartels2009borel} for the $L$-theoretic setting and the $K$-theoretic setting up to dimension $1$ and \cite[Theorem~1.1 and Theorem~3.4]{wegner2012k} for the higher dimensional $K$-theoretic setting.
\end{proof}

\begin{proposition}\label{prop:ICcolimit} Let $(G_i)_{i \in \IN}$ be a directed system of groups indexed over the natural numbers. Suppose that the \FJC~(with finite wreath products) holds for every $G_i$ . Then it also holds for $\colim_{i\in \IN} G_i$.
\end{proposition}
\begin{proof}
First note that $(\colim_{i\in \IN} G_i)\wr F \cong \colim_{i\in \IN} (G_i\wr F)$ for a finite group $F$ and so it suffices to consider the version without wreath products.

This is basically \cite[Theorem~0.7]{Bartels-Echterhoff-Lueck(2008colim)} with the minor problem that this reference does not deal with the version with coefficients in any additive category but only in the category of free $R$-modules for some ring $R$.

First it is shown in \cite[Theorem~3.5]{Bartels-Echterhoff-Lueck(2008colim)} that Isomorphism conjectures are compatible with colimits if the given equivariant homology theory is strongly continuous in the sense of \cite[Definition~2.3]{Bartels-Echterhoff-Lueck(2008colim)}. It is shown in \cite[Lemma~5.2]{Bartels-Echterhoff-Lueck(2008colim)} that $H^?(-;{\bf K}^{alg}_R)$ and $H^?(-;{\bf L}^{\langle-\infty\rangle}_R)$ are strongly continuous for any ring $R$. 

The crucial point is to verify that the canonical maps
\[\colim_i K_n(R\rtimes G_i)\rightarrow K_n(R\rtimes \colim_{i}G_i).\]
and 
\[\colim_i L_n^{\langle-\infty\rangle}(R\rtimes G_i)\rightarrow L_n^{\langle-\infty\rangle}(R\rtimes \colim_{i}G_i).\]
are isomorphisms. The ring $R\rtimes G_i$ denotes the twisted group ring where the $G_i$ action is the restriction of the $\colim_j G_j$-action along the canonical map $G_i\rightarrow \colim_j G_j$. More briefly let us say that the functor $K_n(R\rtimes -)$ is continuous. It is a functor from the category of groups over the group of ring automorphisms of $R$ to the category of abelian groups.

The same statements also hold, if we allow coefficients in any additive category;
A straightforward computation shows the continuity if the functor $\cala\rtimes-$. The continuity of $K_n$ is actually a bit trickier for all $n$. Connective $K$-theory is continuous by construction, however the definition of negative $K$-groups uses a ``delooping'' of additive categories, which is not continuous, but $K$-theory does not see this discontinuity. See for example \cite[Corollary~6.4]{schlichting2006negative}.
In my thesis~\cite[Proposition~10.23]{ruping2013farrell} I also had a straightforward argument that this follows from the long exact sequence associated to a Karoubi-filtration and from the fact, that weak equivalences induce isomorphisms in $K$-theory.  

The proof for the L-theory part from \cite[Lemma~5.2]{Bartels-Echterhoff-Lueck(2008colim)} works also in the setting of additive categories.
\end{proof}

\begin{lemma}\label{lem:FJvAb} The \FJC~(with finite wreath products) holds for any virtually abelian group.
\end{lemma}
\begin{proof} The \FJC~holds for $\IZ^n$ since it is a $\CAT(0)$-group by Theorem~\ref{thm:FJCAT}. Any finitely generated abelian group has a finitely generated, free abelian subgroup of finite index. So the \FJC~with finite wreath products holds for finitely generated abelian groups by Remark~\ref{rem:ICwFOVergroups}. Proposition~\ref{prop:ICcolimit} shows the \FJC~with finite wreath products for abelian groups. Using Remark~\ref{rem:ICwFOVergroups} again, this shows the \FJC~for virtually abelian groups.
\end{proof}

\begin{lemma} Let $F$,$G$  be two groups. If $F$ satisfies the isomorphism conjecture with respect to a family $\calf$ and $G$ satisfies the isomorphism conjecture with respect to a family $\calg$, then $F\times G$ satisfies the isomorphism conjecture with respect to the family \[\calf\times \calg\coloneqq \{H\mid H\le F'\times G'  F\in \calf, G\in \calg \}.\]
\end{lemma}
\begin{proof}
Consider the group homomorphism $p_G:F\times G\rightarrow G \quad (f,g)\mapsto g$. By Proposition~\ref{prop:IChomomrphisms} it suffices to show that 
for any subgroup $H\le G$ with $H\in \calg$ the group $p_G^{-1}(H)=F\times H$ satisfies the isomorphism conjecture relative to the family $\calf \times \calg$. Applying the same argument to the projection $p_H:F\times H\rightarrow H$ it suffices to consider $H'\times H$ with $H'\in \calf,H\in \calg$. This group trivially satisfies the isomorphism conjecture relative to $\calg \times \calf$ since it is an element of the family $\calg \times \calf$.
\end{proof}

\begin{corollary}\label{cor:fjproducts} Let $\calf$ be a class of groups. Suppose that a product of two groups from $\calf$ satisfies the isomorphism conjecture relative to $\calf$. Then the class of groups satisfying the isomorphism conjecture (with finite wreath products) relative to $\calf$ is closed under finite products.

Especially this shows that the class of groups satisfying the \FJC~(with finite wreath products) relative to the family $\VSol$ is closed under finite products, since the class $\VSol$ is. The class of groups satisfying the \FJC~(with finite wreath products) is also closed under finite  products.
\end{corollary}
\begin{proof} Let two groups $G,G'$ be given. Suppose both of them satisfy the isomorphism conjecture relative to the class $\calf$. By the last lemma their product satisfies the isomorphism conjecture relative to the family $\calf \times \calf$. By assumption any group in $\calf\times \calf$ satisfies the isomorphism conjecture relative to $\calf$. So we can reduce the family from $\calf\times \calf$ to $\calf$ by the transitivity principle. 

The version for the wreath products follows from the observation $(G\times G')\wr F\subset (G\wr F)\times (G'\wr F)$.

The final claim follows from the fact that a finite product of virtually cyclic groups is virtually abelian . So it satisfies the \FJC~by Lemma~\ref{lem:FJvAb}.
\end{proof}

\begin{remark}[{\cite[Remark~6.2]{Bartels-Lueck-Reich-Rueping(2012KandL)}}]\label{rem:ICwFOVergroups} Let $\calf$ be a class of groups. Then the class of groups satisfying the isomorphism conjecture with finite wreath products relative to $\calf$ is closed under finite index overgroups.
\end{remark}

We can also combine several of those inheritance properties to get:
\begin{lemma}\label{lem:FJpullbackVC} Let $f:G\rightarrow H$ be a group homomorphism. 
\begin{enumerate}
\item \label{lem:FJpullbackVCeins}If $H$ satisfies the \FJC~and every preimage $f^{-1}(V)$ of a virtually cyclic subgroup $V$ satisfies the \FJC, so does $G$.
\item \label{lem:FJpullbackVCzwei}
If $H$, $\Ker(f)=f^{-1}(1)$ satisfy the \FJC~with finite wreath products and every preimage $f^{-1}(Z)$ of an infinite cyclic subgroup $Z$ satisfies the \FJC~with finite wreath products, so does $G$.
\end{enumerate}
\end{lemma}
\begin{proof}
\begin{enumerate}
\item We know by Proposition~\ref{prop:IChomomrphisms} that $G$ satisfies the \FJC~relative to the family $f^*\VCyc$. Since every group in $f^*\calf$ is a subgroup of a group of the form $f^{-1}(V)$ for some virtually cyclic subgroup $V$ we can apply the transitivity principle (Proposition~\ref{prop:transitivityPrinciple}). So $G$ satisfies the \FJC.
\item By the same argument we have to show that every preimage $f^{-1}(V)$ of a virtually cyclic subgroup $V\subset H$ satisfies the \FJC. If $V$ was finite $f^{-1}(V)$ contains $\Ker(f)$ as a finite index subgroup. The group $f^{-1}(V)$ satisfies the \FJC~with finite wreath products by Remark~\ref{rem:ICwFOVergroups}.

Otherwise $V$ contains an infinite cyclic subgroup $Z$ of finite index. So the index of $f^{-1}(Z)$ in $f^{-1}(V)$ is also finite.  $f^{-1}(V)$ satisfies the \FJC~with finite wreath products by Remark~\ref{rem:ICwFOVergroups}.

\end{enumerate}
\end{proof}

Let us now reduce the family occuring in Theorem~\ref{thm:ICGLNF} to the class of all virtually solvable groups:

\begin{theorem}\label{thm:redtoVsol} Let $V$ be a finitely generated free $Z[S^{-1}]$-module of rank $n$. The group $\aut_{Z[S^{-1}]}(V)$ which is isomorphic to $\GL_n(Z[S^{-1}])$ satisfies the K- and L-theoretic \FJC~with finite wreath products with respect to the class $\VSol$.
\end{theorem}
\begin{proof}
Let $F$  be any finite group. We will show this theorem via induction on $n$. If $\rk(V)=1$ we get that $\aut_{Z[S^{-1}]}(V)\wr F\cong \GL_1(Z[S^{-1}])\wr F$ is virtually abelian. Hence the group itself is virtually solvable. So a point is a model for $E_{\VSol}\GL_1(Z[S^{-1}])\wr F$ and the isomorphism conjecture is trivially true. Let us now consider the case of general $n$:

Let $\calf$ denote the family 
\[\VCyc\cup \{H\mid H\le \Stab(W),W\mbox{ is a nontrivial direct summand of }Z[S^{-1}]\}.\]
We already know that $\GL_n(Z[S^{-1}])\wr F$ satisfies the isomorphism conjecture with respect to the family $\calf^\wr$ by Theorem~\ref{thm:ICGLNF}. Using the transitivity principle we have to show that any group in this family satisfies the \FJC~with finite wreath products relative to the family $\VSol$.

Since the isomorphism conjecture with finite wreath products passes to finite index overgroups by Remark~\ref{rem:ICwFOVergroups}, it suffices to consider 
a product of groups from $\calf$. By Corollary~\ref{cor:fjproducts} we may further restrict to the case of a group $G\in \calf$.

We have to show that $G$ satisfies the \FJC~with finite wreath products for any $G\in \calf$.  If $G$ is virtually cyclic $G\wr F$ is virtually abelian and hence virtually solvable. So the statement is trivial in this case. 

Otherwise $G$ is a subgroup of $\Stab(W)$ for some nontrivial direct summand $W\subset Z[S^{-1}]^n$. So we can assume by Remark~\ref{rem:ICsubgroup} that $G=\Stab(W)$. Let $F$ be any finite group and let $W^\perp$ denote any complement of $W\subset V$. We get an isomorphism $W\oplus W^{\perp}\cong V$ sending $(w,w')$ to $w+w'$. All elements of $\Stab(W)$ have block form with respect to this decomposition. Hence we get a short exact sequence:
\[1\rightarrow \hom_{Z[S^{-1}]}(W^\perp,W)\rightarrow \Stab(W)\stackrel{p}{\rightarrow} \aut_{Z[S^{-1}]}(W)\times \aut_{Z[S^{-1}]}(W^\perp)\rightarrow 1.\]
The map $\Stab(W)\rightarrow \aut_{Z[S^{-1}]}(W)\times \aut_{Z[S^{-1}]}(W^\perp)$  is given by
\[f\mapsto (f|_W,  \pr_{W^\perp}\circ f\circ \inclu_{W^\perp}).\]
The isomorphism from $\hom_{Z[S^{-1}]}(W^\perp,W)$ to $\Ker(p)$  is given by 

\[f\mapsto ((w,w')\mapsto (w+f(w'),w').\]

Applying $-\wr F$ to the epimorphism in the upper short exact sequence we get
\[1\rightarrow \hom_{Z[S^{-1}]}(W^\perp,W)^F\rightarrow \Stab(W)\wr F\stackrel{p}{\rightarrow} (\aut_{Z[S^{-1}]}(W)\times \aut_{Z[S^{-1}]}(W^\perp))\wr F\rightarrow 1.\]

Both factors of $\aut_{Z[S^{-1}]}(W)\times \aut_{Z[S^{-1}]}(W^\perp)$ satisfy the isomorphism conjecture with respect to the family $\VSol$ and hence also $\aut_{Z[S^{-1}]}(W)\times \aut_{Z[S^{-1}]}(W^\perp)$ satisfies the isomorphism conjecture with respect to the family $\VSol$. 
We want to apply Proposition~\ref{prop:IChomomrphisms}. So we have to check that the preimage of any virtually solvable subgroup $H$ of $\aut_{Z[S^{-1}]}(W)\times \aut_{Z[S^{-1}]}(W^\perp)$ satisfies the isomorphism conjecture with respect to the family $\VSol$.

We get a short exact sequence
\[1\rightarrow \hom_{Z[S^{-1}]}(W^\perp,W)^F\rightarrow p^{-1}(H) \stackrel{p}{\rightarrow} H\rightarrow 1.\]
We can identify $\hom_{Z[S^{-1}]}(W^\perp,W)$ with the additive group of $\rk(W)\times \rk(W^\perp)$-matrices over $Z[S^{-1}]$ since $W,W^\perp$ are finitely generated free $Z[S^{-1}]$-modules. Especially it is an abelian group.

The group $p^{-1}(H')$ is a solvable subgroup of $p^{-1}(H)$ of finite index where $H'\le H$ denotes a solvable subgroup of finite index. Hence $p^{-1}(H)$ is also virtually solvable.  So it trivially satisfies the isomorphism conjecture for the family of all virtually solvable subgroups.
This completes the proof. 
\end{proof}

\begin{theorem}\label{ICGlnfullyreduced} $\GL_n(Q)$ and $\GL_n(F(t))$ satisfy the \FJC~with finite wreath products.
\end{theorem}
\begin{proof}
By Proposition~\ref{prop:ICcolimit} it suffices to consider the case, where only finitely many primes are inverted. By the transitivity principle (Proposition~\ref{prop:transitivityPrinciple}) and Theorem~\ref{thm:redtoVsol} it suffices to show the \FJC~(with finite wreath products) for virtually solvable groups. This has recently been done in \cite[Theorem~1.1]{wegner2013farrell}. The finite wreath product version is automatically included since virtually solvable groups are closed under wreath products with finite groups.
\end{proof}

\begin{remark} At the time of writing my thesis the \FJC~was known only for certain class of virtually solvable groups. I needed a lengthy 
argument \cite[8.15-8.21]{ruping2013farrell} that shows that all virtually solvable subgroups of $\GL_n(\IZ)$ and $\GL_n(F(t))$ lie in this class. 
Unfortunately this does not hold for $\GL_n(\IQ)$, so I had a weaker result in my thesis.
\end{remark}

\section{Extensions}

\subsection{Ring extensions}

Again let $Z$ be either $\IZ$ or $F[t]$ for a finite field $F$ and let $S$ be a finite set of primes in $Z$. Assume that we have a ring $R$ and an injective ring homomorphism $i:Z[S^{-1}]\hookrightarrow \Cent(R)\subset R$. This gives $R$ the structure of a left-$Z[S^{-1}]$-module via 
\[(x,r)\mapsto i(x)\cdot r.\] 
In this situation multiplication with an element $r\in R$ is a $Z[S^{-1}]$-linear map. This gives an injective ring homomorphism $f:R \rightarrow \enid_{Z[S^{-1}]}(R)$. Such rings $R$ are also called an associative $Z[S^{-1}]$-algebras.

If we further assume that $R$ is a finitely generated free $Z[S^{-1}]$-module of rank $n$ we have that $\enid_{Z[S^{-1}]}(R)\cong M_n(Z[S^{-1}])$. 

Hence we get an injective ring homomorphism 
\[M_m(R)\rightarrow M_m(M_n(Z[S^{-1}]))\cong M_{mn}(Z[S^{-1}]).\] 
If we restrict to the group of units, we obtain an injective group homomorphism \[\GL_m(R)\rightarrow \GL_{mn}(Z[S^{-1}]).\]
So $\GL_m(R)$ also satisfies the \FJC~with finite wreath products as in Theorem~\ref{ICGlnfullyreduced}.

The following lemma shows that the ring of $S$-integers in a finite field extension of $Q$ has these properties (\cite[Lemma~9.1]{ruping2013farrell}). 

\subsection{Short exact sequences}

The goal of this section is to show the \FJC~for certain extension of groups. Let again $Z$ be either $\IZ$ or the polynomial ring $F[t]$ over a finite field $F$. Let $S$ be a finite set of primes in $Z$.

Let us start with a brief observation:

Let us now consider extensions of $\IZ$ first. The key idea is to use known results about the outer automorphism groups of $\SL_n$ and $\GL_n$.

\begin{lemma}\label{lem:extofZ}Let $n\ge 3$.
\begin{enumerate}
\item \label{lem:extofZeins}Any extension of $\IZ$ by $\SL_n(Z[S^{-1}])$ satisfies the \FJC~with finite wreath products.
\item \label{lem:extofZzwei}Any extension of $\IZ$ by $\GL_n(Z[S^{-1}])$ satisfies the \FJC~with finite wreath products.
\end{enumerate}\end{lemma}
\begin{proof}
\begin{enumerate}
\item Let $G=\SL_n(Z[S^{-1}])\rtimes_\varphi \IZ$ be any such extension. The outer automorphism group of $\SL_n(Z[S^{-1}])$ is torsion (see \cite[Proposition~10.14]{ruping2013farrell}). 
So $G$ has a finite index subgroup isomorphic to $\SL_n(Z[S^{-1}])\times \IZ$. 
$\SL_n(Z[S^{-1}])$ satisfies the \FJC~with finite wreath products by Theorem~\ref{ICGlnfullyreduced} and $\IZ$ does since it is a $\CAT(0)$ group by Theorem~\ref{thm:FJCAT}. 
\item Abelianization induces a group homomorphism \[f:\GL_n(Z[S^{-1}])\rtimes_\varphi \IZ\rightarrow \GL_n(Z[S^{-1}])_{ab}\rtimes_{\varphi_{ab}} \IZ\]
$\GL_n(Z[S^{-1}])_{ab}\rtimes_{\varphi_{ab}} \IZ$ is virtually polycyclic and hence it satisfies the \FJC~with finite wreath products by \cite{Bartels-Farrell-Lueck(2011cocomlat)}[Theorem 0.1]. Note that the class of virtually polycyclic groups is closed under taking wreath products with finite groups. By Lemma~\ref{lem:FJpullbackVC}\ref{lem:FJpullbackVCzwei} it suffices to show that 
the kernel of $f$, which is isomorphic to $SL_n(Z[S^{-1}])$, and every preimage of an infinite cyclic subgroup satisfy the \FJC~with finite wreath products. Theorem~\ref{ICGlnfullyreduced} shows this for $\Ker(f)$ and every preimage of an infinite cyclic subgroup does so by Lemma~\ref{lem:extofZ}\ref{lem:extofZeins}.
\end{enumerate}
\end{proof}

Extensions of $\IZ$ are the basic building blocks for the next proposition.

\begin{proposition}\label{prop:FJextension} Let $n\ge 3$. Suppose that a group $G$ satisfies the \FJC~with finite wreath products. Then
\begin{enumerate}
\item Any extension of $G$ by $\SL_n(Z[S^{-1}])$ satisfies the \FJC~with finite wreath products.
\item Any extension of $G$ by $\GL_n(Z[S^{-1}])$ satisfies the \FJC~with finite wreath products.
\end{enumerate}
\end{proposition}
\begin{proof}
Given such an extension $G'$. Let $f:G'\rightarrow G$ be the projection map. This is just an application of Proposition~\ref{lem:FJpullbackVC}\ref{lem:FJpullbackVCzwei}. 
We have to verify its assumptions. First $p^{-1}(1)$ which is either $\SL_n(Z[S^{-1}])$ or $\GL_n(Z[S^{-1}])$ satisfies the \FJC~with finite wreath products by Theorem~\ref{ICGlnfullyreduced}. Second we have to verify that each preimage of an infinite cyclic group satisfies the \FJC~with finite wreath products. This has been done in Lemma~\ref{lem:extofZ}.
\end{proof}



\begin{thebibliography}{10}

\bibitem{Bartels-Echterhoff-Lueck(2008colim)}
A.~Bartels, S.~Echterhoff, and W.~L\"uck.
\newblock Inheritance of isomorphism conjectures under colimits.
\newblock In Cortinaz, Cuntz, Karoubi, Nest, and Weibel, editors, {\em
  {K}-Theory and noncommutative geometry}, EMS-Series of Congress Reports,
  pages 41--70. European Mathematical Society, 2008.

\bibitem{Bartels-Farrell-Lueck(2011cocomlat)}
A.~Bartels, T.~Farrell, and W.~L{\"u}ck.
\newblock The {F}arrell-{J}ones {C}onjecture for cocompact lattices.
\newblock arXiv:1101.0469v1, 2011.

\bibitem{bartels2009borel}
A.~{B}artels and W.~L{\"u}ck.
\newblock The {B}orel {C}onjecture for hyperbolic and {C}{A}{T} (0)-groups.
\newblock {\em Annals of Mathematics}, 2009.

\bibitem{Bartels-Lueck(2012CAT(0)flow)}
A.~Bartels and W.~L{\"u}ck.
\newblock {G}eodesic flow for {C}{A}{T}(0)-groups.
\newblock {\em Geometry and Topology}, 16:1345--1391, 2012.

\bibitem{Bartels-Lueck-Reich-Rueping(2012KandL)}
A.~Bartels, W.~L{\"u}ck, H.~Reich, and H.~R{\"u}ping.
\newblock {K}- and {L}-theory of group rings over {G}{L}$_n$({Z}).
\newblock {\em arXiv preprint arXiv:1204.2418}, 2012.

\bibitem{Bartels-Reich(2007coeff)}
A.~Bartels and H.~Reich.
\newblock Coefficients for the {F}arrell-{J}ones {C}onjecture.
\newblock {\em Adv. Math.}, 209(1):337--362, 2007.

\bibitem{bridson1999metric}
M.~Bridson and A.~Haefliger.
\newblock {\em {M}etric spaces of non-positive curvature}, volume 319.
\newblock Springer Verlag, 1999.

\bibitem{Bridson-Haefliger(1999)}
M.~R. Bridson and A.~Haefliger.
\newblock {\em Metric spaces of non-positive curvature}.
\newblock Springer-Verlag, Berlin, 1999.
\newblock Die Grundlehren der mathematischen Wissenschaften, Band 319.

\bibitem{garrett1997buildings}
P.~Garrett.
\newblock {\em {B}uildings and classical groups}.
\newblock Chapman \& Hall/CRC, 1997.

\bibitem{Grayson(1984)}
D.~R. Grayson.
\newblock Reduction theory using semistability.
\newblock {\em Comment. Math. Helv.}, 59(4):600--634, 1984.

\bibitem{hatcher2002algebraic}
A.~{H}atcher.
\newblock {\em {A}lgebraic topology}.
\newblock {C}ambridge {U}niversity {P}ress, 2002.

\bibitem{Lueck-Reich(2005)}
W.~L{\"u}ck and H.~Reich.
\newblock The {B}aum-{C}onnes and the {F}arrell-{J}ones conjectures in {$K$}-
  and {$L$}-theory.
\newblock In {\em Handbook of $K$-theory. Vol. 1, 2}, pages 703--842. Springer,
  Berlin, 2005.

\bibitem{mole2013equivariant}
A.~Mole and H.~R{\"u}ping.
\newblock Equivariant {R}efinements.
\newblock {\em arXiv preprint arXiv:1308.2799}, 2013.

\bibitem{Munkres(1975)}
J.~R. Munkres.
\newblock {\em Topology: a first course}.
\newblock Prentice-Hall Inc., Englewood Cliffs, N.J., 1975.

\bibitem{Nagami1962}
K.~Nagami.
\newblock {{M}appings defined on 0-dimensional spaces and dimension theory.}
\newblock {\em J. Math. Soc. Japan}, 14:101--118, 1962.

\bibitem{ruping2013farrell}
H.~R{\"u}ping.
\newblock {\em The {F}arrell-{J}ones conjecture for some general linear
  groups}.
\newblock PhD thesis, Universit{\"a}ts-und Landesbibliothek Bonn, 2013.

\bibitem{schlichting2006negative}
M.~Schlichting.
\newblock Negative {K}-theory of derived categories.
\newblock {\em Mathematische Zeitschrift}, 253(1):97--134, 2006.

\bibitem{wegner2012k}
C.~Wegner.
\newblock The {K}-theoretic {F}arrell-{J}ones conjecture for {C}{A}{T}
  (0)-groups.
\newblock In {\em Proc. Amer. Math. Soc}, volume 140, pages 779--793, 2012.

\bibitem{wegner2013farrell}
C.~Wegner.
\newblock The {F}arrell-{J}ones {C}onjecture for virtually solvable groups.
\newblock {\em arXiv preprint arXiv:1308.2432}, 2013.

\end{thebibliography}

\begin{acknowledgements}\label{ackref}
I would like to thank my advisor Wolfgang L\"uck for his support and encouragement, Arthur Bartels, Holger Reich for all discussions on the \FJC. Special thanks goes to Adam Mole for removing a condition on the bound of the order of finite subgroups. Furthermore I want to thank Philipp K\"uhl, Christian Wegner for proofreading. 
\end{acknowledgements}

\affiliationone{
  H. R\"uping\\
  Endenicher Allee 60, 53115 Bonn\\    
   Germany
   \email{henrik.rueping@hcm.uni-bonn.de}
   }
\end{document}